\documentclass[11pt,oneside,reqno]{amsart}

\usepackage{mathrsfs} 
\usepackage{amsmath}
\usepackage{hyperref}
\usepackage{amsmath}
\usepackage{amssymb}
\newtheorem{theorem}{Theorem}[section]

\newtheorem{proposition}{Proposition}[section]
\newtheorem{corollary}{Corollary}[section]
\newtheorem{lemma}{Lemma}[section]

\newtheorem{remark}{Remark}[section]
\numberwithin{equation}{section}

\def \Z{\mathbb{Z}}
\begin{document}
\title{On Lehmer's Problem}
\author[Piyush Kumar Jha]{Piyush Kumar Jha}
\author[R. Thangadurai]{R. Thangadurai$^{1,2}$}
\address[Piyush Kumar Jha]{3rd year, B.Sc (Mathematics), Chennai Mathematical Institute, H1, SIPCOT IT Park, Siruseri, Kelambakkam 603103, India  }
\address[R. Thangadurai]{${}^{1}$Harish-Chandra Research Institute,   Chhatnag Road,  Jhunsi, Prayagraj 211019, India.}
\address{ ${}^{2}$ Homi Bhabha National Institute, Training School Complex, Anushaktinagar, Mumbai, 400094  India. }
 
\email[Piyush Kumar Jha]{piyushj.ug2024@cmi.ac.in}
\email[R. Thangadurai]{thanga@hri.res.in}

\bigskip
\begin{abstract}
Let $m\geq 2$ be an integer, and let $d\geq 2$ be an even integer. We study monic, reciprocal, non-cyclotomic polynomials
$$
P(x)=x^d+a_{d-1}x^{d-1}+\cdots+a_1x+1\in\mathbb{Z}[x].
$$
For a fixed residue class $r\in\{0,1,\ldots,m-1\}$, we prove that there exists an effectively computable constant $c_{r,m}>1$ such that the Mahler measure 
$$
M(P)\geq c_{r,m}
$$
for every irreducible polynomial $Q$ whose inner coefficients satisfy
$$
a_j\equiv r\pmod m \qquad (1\leq j\leq d-1).
$$
Moreover, we explicitly compute the constant $c_{r,m}$ in terms of $r$ and $m$. For the residue classes $r=2$, $r=3$, and $r=m-1$, we obtain improved lower bounds for suitable ranges of the modulus~$m$.

We also consider monic, reciprocal, irreducible, non-cyclotomic polynomials whose inner coefficients belong to two distinct residue classes modulo $m$. More precisely, if the inner coefficients with odd indices are congruent to $a\pmod m$ and those with even indices are congruent to $b\pmod m$, where $a\not\equiv b\pmod m$, then we prove the existence of a constant $c_{m,a,b}>1$ such that the Mahler measure 
$$
M(P)\geq c_{m,a,b}.
$$
In particular, for $m=5$, $a=4$, and $b=2$, we obtain
$$
c_{5,4,2}\geq \left(\frac{976}{975}\right)^{1/3}.
$$
These results solve Lehmer's problem for several families of polynomials. We also construct infinitely many examples in every even degree $d\geq 2$ of monic, reciprocal, irreducible, non-cyclotomic polynomials over $\mathbb{Q}$, together with explicit information about their Mahler measures.

Our proofs rely on Y. Bilu's equidistribution theorem and its effective version proved by C. D'Andrea \emph{et al.}, as well as on a method developed by P. Borwein \emph{et al.}   While it is natural to consider Bilu's equidistribution results in conjunction with the Borwein {\it et al}. method to Lehmer's problem, we are not aware of any previous work in the literature that employs this particular approach.  
 \end{abstract}
\maketitle

\section{Introduction}

For a monic polynomial
$$
P(x)=x^d+a_{d-1}x^{d-1}+\cdots+a_1x+a_0\in\mathbb{Z}[x],
$$
the Mahler measure  $M(P)$ is defined by
$$
M(P)=\prod_{P(\alpha)=0}\max\{1,|\alpha|\},
$$
where the product is taken over all roots of $P$, counted with multiplicity.

By Kronecker's theorem, if $P\in\mathbb Z[x]$ is
monic, irreducible, and satisfies $P(0)\ne0$,
then $M(P)=1$ if and only if $P$ is cyclotomic. In 1933, D. H. Lehmer \cite{leh} asked whether, for every $\epsilon>0$, there exists a polynomial $P\in\mathbb{Z}[x]$ such that
$1<M(P)<1+\epsilon.$ This question is now known as Lehmer's problem.

Lehmer observed that the polynomial $P(x)=x^{10}+x^9-x^7-x^6-x^5-x^4-x^3+x+1$ has Mahler measure
$$
M(P)=1.176280\ldots.
$$
This remains the smallest known Mahler measure greater than $1$ for a non-cyclotomic polynomial with integer coefficients.

Lehmer's problem has been resolved for several special classes of polynomials. For example, in 1971, C. J. Smyth \cite{smyth01} proved that if $P(x)\in\mathbb{Z}[x]$ is non-reciprocal (even up to sign) and $P(0)\neq 0$, then
$$
M(P)\geq M(x^3-x-1)=1.324717\ldots.
$$
In 2007, P. Borwein \emph{et al.} \cite{borw} proved that if $P(x)\in\mathbb{Z}[x]$ is a monic polynomial of degree $d$ with no cyclotomic factors whose coefficients are all odd, then
$$
\log M(P)\geq
\frac{\log 5}{4}\left(1-\frac{1}{d+1}\right).
$$
In 2012, J. H. Silverman \cite{silver} proved that, for every $\epsilon>0$, there exists a constant $c_\epsilon>0$ such that the following holds. If $P\in\mathbb{Z}[x]$ is a monic polynomial of degree $d$ with no cyclotomic factors and $x^{n-1}+x^{n-2}+\cdots+x+1$
divides $P(x)$ in $(\mathbb{Z}/m\mathbb{Z})[x]$, where $m\geq 2$ and $n\geq \max\{\epsilon d,2\},$
then $\log M(P)\geq c_\epsilon\log m.$ Numerous further results are known under additional assumptions concerning the arithmetic properties of the associated number fields; see, for example, \cite{amo2,boy,fla,gar,gor1,gor,gor2,mig,mos1,luc,rhi,saunders,vou}.

Throughout this article, we consider monic reciprocal polynomials $P(x)\in\mathbb{Z}[x]$  with no cyclotomic factors. Since a reciprocal polynomial with constant term $1$ and no cyclotomic factors has even degree unless it is of degree $0$, the degrees considered below are even. Our main objective is to obtain explicit lower bounds for the Mahler measure of such polynomials under prescribed congruence conditions on their coefficients.

We first consider the case in which all the non-leading and non-constant coefficients belong to the same residue class modulo $m$.

\begin{theorem}\label{thm9.1}
Let $m\geq 2$ be an integer, and let $r\in\{0,1,\ldots,m-1\}.$ Let $\mathscr{P}_{m,r}$ denote the collection of all monic, reciprocal, irreducible, non-cyclotomic polynomials
$$
P(x)=x^d+a_{d-1}x^{d-1}+\cdots+a_1x+1\in\mathbb{Z}[x],
$$
of degree $d\geq 2$, satisfying $a_j\equiv r\pmod m$ for $1\leq j\leq d-1.$ Then there exists a computable constant $c_{r,m}>1$ such that
$$
\inf_{P\in\mathscr{P}_{m,r}}M(P)\geq c_{r,m}>1.
$$
\end{theorem}

We compute the constant $c_{r,m}$ explicitly using the effective version of equidistribution result of C. D'Andrea \emph{et al.}, in the following corollary.

\begin{corollary}\label{thm9.1-cor1}
Under the hypotheses of Theorem \ref{thm9.1}, for $1\le r\le m-1$ let $c\equiv r-1\pmod m$ be chosen with $\delta=|c|=\min_{t\in\mathbb{Z}}|r-1+tm|$ (the case $r=0$ is treated separately). Then we have 
\begin{equation}\label{eq:6.14}
c_{r,m}=\left\{
 \begin{array}{ll}
 {\sqrt{m^2+1}\over2},&r=0,\\[5pt]
 \left(\frac{\sqrt{m^2+1}}{2}\right)^{2/3},&r=1,(\mbox{P. Borwein et al. \cite{borw}})\\[5pt]
 m^{1/123341},&r=2\text{ and }m\geq3, (\mbox{J. H. Silverman \cite{silver}})\\[5pt]
 \left(\frac{\sqrt{(m/2)^2+1}}{2}\right)^{2/3},&m\geq4\text{ even and }\delta=m/2,\\[5pt]
 \exp(\eta(m,\delta)),&\delta\geq2\text{ and }2\delta<m.
 \end{array}\right.
\end{equation}
Here  $\eta(m, \delta)$ is defined as 
\begin{equation}\label{eq:6.7}
\eta(m,\delta)=\min\left\{
 \log{3\over2},\ \log{3+\sqrt5\over2},\ V(D_0),\
 {\Delta\over40},\ {9\Delta^2\over128K_\delta^2}
 \right\}
\end{equation}
where 
\begin{equation}\label{eq:6.4}
\Delta=\log{m\over2\delta}>0,\qquad
 K_\delta={36\pi\over\delta-1}
\end{equation}
and $D_0=D_0(m,\delta)$ is the least even integer $D_0\geq4$ satisfying
\begin{equation}\label{eq:6.5}
8K_\delta\sqrt{{\log(D_0+1)\over D_0}}\leq{\Delta\over4}
\end{equation}
and  by setting 
\begin{equation}\label{eq:6.6}
V(D_0)=\min_{\begin{array}{c}4\leq d<D_0\\d\text{ even}\end{array}}
 {1\over4}\left({\log\log d\over\log d}\right)^3,
\end{equation}
with the minimum of the empty set interpreted as $+\infty$.
\end{corollary}
\begin{remark}
In the definition of $D_0 = D_0(m,\delta)$, since $\log (x+1)/x$ is a decreasing function for all $x\geq 2$, we see that  \eqref{eq:6.5} holds for every integer $d\geq D_0$. 
\end{remark}
For the residue classes $r=2,3$, and $m-1$, we obtain stronger estimates.

 \begin{theorem}\label{thm:2,3,-1-residues}
Let $m, d\ge 2$ be integers. Let $P(x)=x^d+a_{d-1}x^{d-1}+\cdots+a_1x+1\in \mathbb{Z}[x]$ be a monic, reciprocal, irreducible polynomial which 
satisfies $a_j\equiv r\pmod m$ for $1\leq j\leq d-1$. Suppose that $P$ has no cyclotomic factors. Then the following estimates hold whenever their stated conditions are satisfied; in overlapping cases, each applicable estimate holds.
$$
M(P) \geq \left\{\begin{array}{ll}
\exp\left(\frac{3}{19}\left(\frac{\log2}{2}-\frac{\log(19/2)+1}{17}\right)\right) \geq 1.0248275855\ldots  & r = 2, m\geq 3\\
 \left(\frac{3\sqrt3}{4}\right)^{d/(3d+3)} \geq \left(\frac{3\sqrt3}{4}\right)^{2/9} & r = 2, m = 3\\
 \left(\frac{m}{4}\right)^{d/(d+1)}\geq \left(\frac{m}{4}\right)^{2/3} & r= 2, m\geq 5\\
 2^{3/7} & r = 3, m =5\\
 2^{d/(2d+2)} & d\geq 6, r = 3, m = 5\\
  \min\left\{2, \left(\frac{m}{6}\right)^{d/(d+1)}\right\}\geq \min\left\{2,\left(\frac{m}{6}\right)^{2/3}\right\}& r =3, m\geq 7\\
   2^{3/7}& r = 4, m = 5\\
    2^{d/(2d+2)} \geq 2^{3/7}& d\geq 6, r = 4, m = 5 \\
    \min \left\{\rho,\left(\frac{m}{6}\right)^{2/3}\right\}& r = m-1, m\geq 7, 
   \end{array}\right.
   $$
   where $\rho$ is the largest real root of $x^4-x^3-x^2-x+1$. 
   \end{theorem} 
 \begin{remark} 
 The pairs $(r,m)$ not listed in Theorem~\ref{thm:2,3,-1-residues} are covered by Corollary~\ref{thm9.1-cor1}. For example, for $r=3$ and $m=6$, reduction modulo $3$ places $P$ in the class $r=0$ modulo $3$, and Lemma~ \ref{lem:zero} gives $M(P)\ge\sqrt{10}/2$; for $r=m-1$ with $m=4$ we have $\delta=2=m/2$, and for $r=m-1$ with $m=6$ we have $\delta=2$ and $2\delta<m$. 
 \end{remark}
 We next consider the case in which the inner coefficients belong to two distinct residue classes modulo $m$.

\begin{theorem}\label{thm:abmodm}
 Let $m\geq 3$ be an integer. Let $0\le a,b\le m-1$ with $a\ne b$. For each integer $d\geq1$, let $\mathscr{P}_{m,a,b}(d)$ be the collection of all monic, reciprocal and irreducible polynomials $P(x) = x^{2d} + c_{2d-1}x^{2d-1}+\cdots+c_2x^2+c_1x+1\in \mathbb{Z}[x]$ that are non-cyclotomic such that $c_j = c_{2d-j}$, $c_{2j+1} \equiv a\pmod{m}$ for $0\leq j\leq d-1$ and $c_{2j}\equiv b\pmod{m}$ for $1\leq j\leq d-1$. Set $\mathscr{P}_{m,a,b}=\bigcup_{d\geq1}\mathscr{P}_{m,a,b}(d)$. Then we have 
 $$
 \inf_{P\in \mathscr{P}_{m,a,b}} M(P) > 1.
 $$
 In other words, there exists a constant $c_{m,a,b} > 1$, depending only on $m, a$ and $b$ such that  every $P\in \mathscr{P}_{m,a,b}$ satisfies $M(P) \geq c_{m,a,b}$. In particular, we have  
 $$
c_{5,4,2} \geq  \left(\frac{976}{975}\right)^{1/3}.
 $$
 \end{theorem}

The following theorem provides infinitely many examples in every even degree and supports Theorem \ref{thm9.1}.

\begin{theorem}\label{thm:general-irreducibility}
Let $d\geq 2$ be an even integer and let $A\geq 4$ be an integer. Define
$$
P_{d,A}(x)=x^d+Ax^{d-1}+\cdots+Ax+1.
$$
Then the following statements hold:
\begin{enumerate}
\item $P_{d,A}$ is irreducible over $\mathbb{Q}$ and non-cyclotomic. Moreover, exactly $d-2$ of its roots lie on the unit circle, while the remaining two roots are real, negative, and reciprocal; one lies outside the unit circle and the other lies inside it.

\item We have $A-1<M(P_{d,A})<A.$
Furthermore,
$$
\lim_{\substack{d\to\infty, d \text{ even}}}M(P_{d,A})=A-1 \mbox{ and  consequently, }
\inf_{\substack{d\geq 2, d \text{ even}}}M(P_{d,A})=A-1.
$$
\end{enumerate}
\end{theorem}
Finally, for integers $A\ge4$ and $B\ge1$ with $A>B$, we consider the polynomials
$$
P_{d,A,B}(x)=x^{2d}+Ax^{2d-1}+Bx^{2d-2}+\cdots+Bx^2+Ax+1,
$$
whose odd and even inner coefficients lie in the residue classes $A$ and $B$ modulo any $m$ with $A\not\equiv B\pmod m$; they therefore belong to $\mathscr{P}_{m,a,b}$ with $a\equiv A$, $b\equiv B\pmod m$ whenever they are irreducible and non-cyclotomic. The next theorem provides a family of such examples. 
\begin{theorem}\label{thm:support-abmodm}
Let   $d\geq 2$ be an even integer. Let $A, B\in \mathbb{Z}$ with $A>B$ such that $A\geq 4 , B\geq 1$ and $\gcd(A-1, B-1) \geq 3$. Then we prove the following. 
\begin{enumerate} 
\item $P_{d, A, B}(x) \in \mathbb{Z}[x]$, as defined above, is irreducible over $\mathbb{Q}$ and non-cyclotomic. Moreover, exactly $2d-2$ roots lie on the unit circle, and the remaining two roots lie off the unit circle; they are real, negative and reciprocal to each other. 
\item We have $A-1 < M(P_{d, A, B}) < A$. Moreover, if $\displaystyle R_{A,B} = \frac{A+\sqrt{A^2-4(B-1)}}{2}$, then 
$$  \lim_{\substack{d\to\infty, \ d \ \text{even}}} M(P_{d, A, B}) = R_{A, B} \mbox{ and consequently, } A-1 < R_{A,B} \le A
$$ 
and $R_{A,1} = A$. 
\end{enumerate}
\end{theorem}
\section{Preliminaries}

For a polynomial $\displaystyle F(x)=\sum_{j=0}^{m}c_jx^j$, the length of $F$ is defined as  
$$
\displaystyle L(F) :=\sum_{j=0}^{m}|c_j|\mbox{ and }\displaystyle\|F\|_{\infty}:=\max_{|z|=1}|F(z)|.
$$

\begin{lemma} (See Lemma 3.2 in \cite{borw})
\label{lem:length-limit}
For every nonzero polynomial $F\in\mathbb{C}[x]$, $$\lim_{k\to\infty}L(F^k)^{1/k}=\|F\|_{\infty}.$$
\end{lemma}


Using Lemma \ref{lem:length-limit}, we can prove the following.
\begin{lemma}
\label{lem:resultant-mahler} Let $Q\in\mathbb{Z}[x]$ be a monic polynomial of degree $n$, and let $H\in\mathbb{C}[x]$ be a polynomial of degree $h$. Then 
$$
\left|\operatorname{Res}(Q,H)\right|\le \|H\|_{\infty}^{n}M(Q)^{h},
$$
where $\operatorname{Res}(Q, H)$ denotes the resultant of $Q$ and $H$. 
\end{lemma}
\begin{proof}
Let $\alpha_1,\ldots,\alpha_n$ be the roots of $Q$. For each integer $k\geq1$, the elementary length estimate gives
$$
|\operatorname{Res}(Q,H)|^k=\prod_{j=1}^n|H^k(\alpha_j)|
\leq L(H^k)^n M(Q)^{kh}.
$$
Taking $k$-th roots and then letting $k\to\infty$ proves the assertion by Lemma~\ref{lem:length-limit}.
\end{proof}
For a real number $u$ with $|u|<1$, define {\it the Blaschke factor} 
$$b_u(z)=\frac{z-u}{1-uz}.$$ 
Note that whenever $|z|=1$, we have $|b_u(z)|=1.$

\begin{lemma}\label{lem:blaschke-derivative}
Let $|u|<1$ be a real number, and choose a continuous argument $b_u(e^{it})=e^{i\theta_u(t)}.$ Then 
$$\theta_u'(t) = \frac{1-u^2}{1-2u\cos t+u^2}.
$$ 
Moreover, $\theta_u(2\pi)-\theta_u(0)=2\pi.$
\end{lemma}

\begin{proof}
A direct logarithmic differentiation implies 
$$\frac{d}{dt}\log b_u(e^{it}) = \frac{i e^{it}}{e^{it}-u} + \frac{i u e^{it}}{1-u e^{it}}.$$ Taking imaginary parts and simplifying yields $$\theta_u'(t) = \frac{1-u^2}{1-2u\cos t+u^2}.$$ The factor $b_u$ has one zero and no pole in the open unit disk, and hence its restriction to the unit circle has winding number $1$.  Hence its argument increases by $2\pi$.
\end{proof}


\begin{proposition}
\label{prop:root-distribution}
Let $N\geq 2$ be an even integer and $A\geq 4$ be an integer. 
The polynomial $P_{N,A} = x^N + Ax^{N-1} +Ax^{N-2}+\cdots+Ax+1$ has exactly $N-2$ roots on the unit circle and exactly two roots lie off the unit circle. These two roots are real, negative, and reciprocal to each other.
\end{proposition}

\begin{proof}
Put $c = A-1$ and hence $c\geq 3$. Now consider 
$$
F_{N,c}(x) = (x-1) P_{N,A}(x) = x^{N+1} +cx^N-cx-1.
$$
Then the equation 
$$F_{N,c}(z)=0 \iff z^N(z+c)=cz+1,$$ 
and hence, for $z\neq -c$, 
$$
z^N=\frac{cz+1}{z+c}.
$$ 
For $|z|=1$, define 
$$
B(z)=\frac{cz+1}{z+c}=\frac{z+c^{-1}}{1+c^{-1}z}.
$$ 
Since $c^{-1}\in(0,1)$, this is a Blaschke factor. In particular, 
$$
|B(z)|=1 \mbox{ for all } |z|=1.
$$ 
Write 
$$
B(e^{it})=e^{i\beta(t)} \mbox{ with } \beta(0)=0,
$$ 
where $\beta$ is chosen continuously. Since $B$ has one zero at $z = -c^{-1}$ and no pole in the open unit disk, $\beta(2\pi)=2\pi.$ By differentiating it, we get 
\begin{equation}\label{eq:beta-derivative} 
\beta'(t)=\frac{c^2-1}{c^2+2c\cos t+1}.
 \end{equation} 
 Since $(c-1)^2 \leq c^2 + 2c\cos t+1\leq (c+1)^2$,  consequently, we obtain
 $$\frac{c-1}{c+1}\le \beta'(t)\le \frac{c+1}{c-1}.
 $$ 
 Since $c\ge 3$, we see that $\displaystyle\frac{c+1}{c-1}\le 2.$ By setting $H(t)=Nt-\beta(t)$, as $N\geq 2$,  we see that $H'(t)=N-\beta'(t)\ge 0.$ The above inequality is strict for every $t$, except for the case when $(N, c) = (2, 3)$. When $(N, c) = (2,3)$, we see that $H'(t)$ vanishes only at $t=\pi$ and hence $H$ is still strictly increasing on $[0,2\pi]$. Moreover, $H(0)=0 \mbox{ and } H(2\pi)=2\pi(N-1).$ Therefore, when $0\le t<2\pi$, each equation $$H(t)=2\pi k,\qquad k=0,1,\ldots,N-2,$$ has exactly one solution. Thus $F_{N,c}$ has exactly $N-1$ distinct roots on the unit circle.

Since $F_{N,c}(1)=0$, one of these roots is $1$. By the definition of $F_{N,c}(x)$, we conclude that the remaining $N-2$ roots, that are lying on the unit circle, are the roots of $P_{N,A}$. Since $N$ is even, 
\begin{equation}\label{eq:anti-reciprocal} 
x^{N+1}F_{N,c}(1/x)=-F_{N,c}(x). 
\end{equation} 
Hence the roots of $F_{N,c}$ occur in reciprocal pairs. Furthermore, $F_{N,c}(-c)=c^2-1>0,$ whereas 
$F_{N,c}(x)\longrightarrow -\infty \mbox{ as } x\to -\infty,$ because $N+1$ is odd. Hence $F_{N,c}$ has a real root $\rho$ satisfying $\rho<-c<-1.$ By \eqref{eq:anti-reciprocal}, $\rho^{-1}\in(-1/c,0)$ is also a root.

Now, we have accounted for $N-1+2 = N+1$ roots of $F_{N,c}$ which is a polynomial of degree $N+1$. Thus, by removing the root $1$,  we conclude that $P_{N,A}$ has $N-2$ roots on the unit circle and exactly two roots off the unit circle, which are namely $\rho$ and $\rho^{-1}$. This proves the proposition. 
\end{proof}

The following lemma is useful in proving the irreducibility. 

\begin{lemma}
\label{lem:root-unity}
Let $K$ be a number field, and $\mathcal{O}_K$ be the ring of integers of $K$.  If $1-\xi\in c\mathcal{O}_{K}$ for some root of unity $\xi \in \mathcal{O}_K$, and for some integer $c\geq 3$,  then $\xi=1$.
\end{lemma}

\begin{proof}
Given that $\xi\in \mathcal{O}_K$ is a root of unity.  Suppose $\xi\neq 1$, and let $n>1$ be its order. Let $L=\mathbb{Q}(\xi) \subset K$ and let the degree be $e=[K:L]$.  Since $1-\xi\in c\mathcal{O}_{K}$, there exists $\alpha\in\mathcal{O}_{K}$ such that 
$$
1-\xi=c\alpha.
$$ 
By taking norms over $K/\mathbb{Q}$, we obtain 
$$
N_{K/\mathbb{Q}}(1-\xi)=c^{[K:\mathbb{Q}]}N_{K/\mathbb{Q}}(\alpha).
$$ 
Since $1-\xi\in L$, we get 
$$
N_{K/\mathbb{Q}}(1-\xi)=N_{L/\mathbb{Q}}(1-\xi)^e.
$$ 
Now 
$$
\bigl|N_{L/\mathbb{Q}}(1-\xi)\bigr|=\Phi_n(1),
$$ where $\Phi_n$ is the $n$th cyclotomic polynomial. Therefore 
\begin{equation}\label{eq:norm-divisibility}
c^{e\varphi(n)}\mid \Phi_n(1)^e.
\end{equation}
For $n>1$, one has
\[
\Phi_n(1)=
\begin{cases}
p, & n=p^r\text{ for some prime }p,\\
1, & n\text{ is not a prime power}.
\end{cases}
\]
If $n$ is not a prime power, then \eqref{eq:norm-divisibility} would imply $c^{e\varphi(n)}\mid 1$, which is not possible.  Thus $n=p^r$ for some prime $p$, and hence 
$$
c^{e\varphi(p^r)}\mid p^e.
$$ 
It follows that $c=p^s$ for some integer $s\ge 1$, and by comparison of $p$-adic valuations, we obtain
$$
s\varphi(p^r)\le 1.
$$ 
Hence $s=1$ and $\varphi(p^r)=1$. The only possibility is $$p=2,\qquad r=1,\qquad c=2,$$ contrary to $c\ge 3$. Thus we get $\xi=1$.
\end{proof}

The following lemma  is proved   in \cite{silver}.
\begin{lemma} \label{lem:complex} (See Proposition 2.1 in \cite{silver}) 
Let $J\geq1$ be an integer. For every complex number $z$ satisfying $|z|\leq1$, we have 
$$
\sum_{j=1}^J\left(1-\frac{j}{J+1}\right)\log|1-z^j|\leq\frac12\log\left(\frac J2+1\right)+\frac12.
$$
\end{lemma}

We require the following equidistribution results.

\begin{theorem}\label{bilu} (Y. Bilu, \cite{bilu}) Let $(\xi_n)_n$ be a strict sequence (A sequence $(\xi_n)$ in $\mathbb G_m^N$ is strict if, for every proper algebraic subgroup $H$ of $\mathbb G_m^N$, only finitely many $\xi_n$ lie in $H$) in $(\overline{\mathbb{Q}}^\times)^N$ such that $h(\xi_n)\to 0$ as $n\to\infty$, where $h$ is the absolute logarithmic Weil height. If $S_n$ is the Galois orbit of $\xi_n$ and $\mu_{S_n}$ is the uniform probability measure on $S_n$, then $\mu_{S_n}$ converges weakly to a Haar probability measure on the unit polycircle $(\mathbb{S}^1)^N$. 
\end{theorem}

More precisely, when $N=1$, we let 
$$
\mu_n = \frac{1}{d_n}\sum_{P_n(\alpha) = 0} \delta_\alpha
$$
where $\alpha\in \mathbb{C}$ and $P_n$ is the minimal polynomial for $\xi_n$ of degree $d_n$. Thus $\mu_n$ is the discrete probability measure on $\mathbb{C}$ supported equally on $\xi_n$ and its Galois conjugates. When $N = 1$, Theorem \ref{bilu} says that $\mu_n \to \mu_{\mathbb{S}^1}: = \frac{m_{\mathbb{S}^1}}{2\pi}$ in the weak-${}^*$ topology on $\mathbb{C}$. In other words, given any continuous function $\varphi$ with compact support on $\mathbb{C}$, we have 
\begin{equation}\label{eq2.10000}
\frac{1}{d_n}\sum_{P_n(\alpha) = 0}\varphi(\alpha) = \int_\mathbb{C} \varphi d\mu_n \to \frac{1}{2\pi}\int_{0}^{2\pi}\varphi(e^{it}) dt.
\end{equation}

The following is an effective version of Theorem \ref{bilu} when $N =1$, which is useful to prove Corollary \ref{thm9.1-cor1}.
For $z=e^{u+2\pi i\theta}$ and $w=e^{v+2\pi i\phi}$ with $u,v\in\mathbb{R}$ and $\theta,\phi\in\mathbb{R}/\mathbb{Z}$, we use the logarithmic-polar metric
$$
d_{\mathrm{lp}}(z,w)=\left((u-v)^2+\operatorname{dist}_{\mathbb{R}/\mathbb{Z}}(\theta,\phi)^2\right)^{1/2},
$$
where $\operatorname{dist}_{\mathbb{R}/\mathbb{Z}}(\theta,\phi)=\min_{n\in\mathbb{Z}}|\theta-\phi+n|$.
In Theorem \ref{sombra} and below, $\operatorname{Lip}(F)$ denotes the Lipschitz constant with respect to this metric.
\begin{theorem}\label{sombra} (C. D'Andrea \emph{et al.}, \cite{sombra}) 
Let $\xi\in\overline{\mathbb{Q}}^{\times}$ be an algebraic number of degree $d$ and height $h(\xi)\leq1$,
and let $S$ be its Galois orbit.  Let $F:\mathbb{C}^*\to\mathbb{R}$ be a Lipschitz function and its 
restriction $F_0$ to $\mathbb{S}^1$ is of class $C^2$.  Then
\begin{equation}\label{eq:2.1}
\left|{1\over d}\sum_{\beta\in S}F(\beta)
       -\int_{\mathbb{S}^1}F\,d\lambda\right|
 \leq c(F)\left(4h(\xi)+64{\log(d+1)\over d}\right)^{1/2},
\end{equation}
where $\lambda$ is Haar probability measure on $\mathbb{S}^1$, and one may take
\begin{equation}\label{eq:2.2}
c(F)\leq2\operatorname{Lip}(F)+16\|\widehat{F_0'}\|_1.
\end{equation}
Here the prime means differentiation in the angular coordinate
$\theta\in\mathbb{R}/\mathbb{Z}$, and the norm in \eqref{eq:2.2} is the $\ell^1$-norm of
the Fourier coefficients.
\end{theorem}
   The absolute numerical constant inside the square root in \eqref{eq:2.1} may be taken to be $64$, while $c(F)$ is controlled by \eqref{eq:2.2}. In dimension one, the generalized degree in \cite{sombra} is the usual degree. The following known results are useful to prove our results. 
 
 \begin{theorem}\label{voutier}(P. Voutier, \cite{vou})
 If $\alpha$ is a non-zero algebraic integer of degree $d\geq 2$ which is not a root of unity, then 
 $$
 \log M(\alpha)  > \frac{1}{4}\left(\frac{\log\log d}{\log d}\right)^3.
 $$
 \end{theorem}
 
 \begin{theorem}\label{borwein} (P. Borwein \emph{et al.} \cite{borw})
 Let $m\geq 2$, and let $P(x)\in \mathbb{Z}[x]$ be a monic polynomial of degree $d$ with no cyclotomic factors that satisfies
 $P(x) \equiv x^d+x^{d-1}+\cdots+x+1\pmod{m}$. Then 
 $$
 \log M(P) \geq \frac{d}{d+1} c_m$$
 where  $c_2 = (\log 5)/4$ and $c_m = \log(\sqrt{m^2+1}/2)$ for $m\geq 3$. 
 \end{theorem}
 
 \begin{theorem}\label{sliver} (J. Silverman, \cite{silver}) 
 For all $\epsilon >0$, there is a constant $c_\epsilon > 0$ with the following property: Let $P(x) \in \mathbb{Z}[x]$ be a monic polynomial of degree $d$ such that $P(x)$ is divisible by $x^{n-1}+x^{n-2}+\cdots+x+1$ in $(\mathbb{Z}/m\mathbb{Z})[x]$ for some integers $m\geq 2$ and $n \geq \max\{\epsilon d, 2\}$. Suppose further that no root of $P(x)$ is a root of unity. Then 
 $$
 \log M(P) \geq c_\epsilon \log m.
 $$
 \end{theorem}
 
 The following proposition is useful to prove Theorems \ref{thm9.1} and \ref{thm:abmodm} using the equidistribution theorem of Y. Bilu (Theorem \ref{bilu}). We need the standard complex analysis lemma as follows.

\begin{proposition}\label{pro2.10000}
Let $\rho_0 < \rho_1 < \rho_2 < \rho_3$ be positive real numbers and $S = \{z\in\mathbb{C}\ : \ \rho_1 \leq |z|\leq \rho_2\}$ be the annulus in $\mathbb{C}$. Then there exists a continuous function $f:\mathbb{C}\to \mathbb{R}$ such that 
$$
f(z) = \left\{\begin{array}{ll}
0, & |z| \leq rho_0\\
\frac{|z|-\rho_0}{\rho_1-\rho_0}, & \rho_0 < |z| <\rho_1\\
1, & \rho_1\leq |z| \leq \rho_2\\
\frac{\rho_3-|z|}{\rho_3-\rho_2},& \rho_2 < |z| < \rho_3\\
0, &|z|\geq \rho_3.
\end{array}\right.
 $$
 \end{proposition}

Let $\mathscr{P}$ be the collection of all monic, irreducible, reciprocal and non-cyclotomic polynomials with integer coefficients of degree $d$ for any integer $d\geq 2$. Then we prove the following.

\begin{theorem}\label{thm1}
Let $\mathscr{S} \subset \mathscr{P}$ be an infinite subset of $\mathscr{P}$. 
If $\displaystyle \inf_{P\in\mathscr{S}} M(P) = 1$, then there exists a sequence of pairwise distinct polynomials $\{P_\nu\}_\nu$ in $\mathscr{S}$ of degree $d_\nu$ satisfying the following properties. 
\begin{enumerate}
\item[(1)] As $\nu\to\infty$, we have $d_\nu \to \infty$.
\item[(2)] For any continuous function $\varphi$ with compact support on $\mathbb{C}$, we have 
$$
\lim_{\nu\to\infty}\frac{1}{d_\nu}\sum_{\alpha\in \mathbb{C}, P_\nu(\alpha) =0} \varphi(\alpha) =\frac{1}{2\pi}\int_0^{2\pi}\varphi(e^{it}) dt.
$$
\item[(3)] For any continuous function $\varphi$ with a compact support on $\mathbb{C}^2$, we have 
$$
\lim_{\nu\to\infty}\int_{\mathbb{C}^2} \varphi d\mu_\nu = \int_{\mathbb{S}^1\times \mathbb{S}^1} \varphi d\lambda
$$
where $\mu_\nu$ is the uniform probability measure on the Galois orbits of $(\alpha_\nu, \alpha_\nu^{d_\nu})$ for a root $\alpha_\nu$ of $P_\nu$ of degree $d_\nu$ and $\lambda$ is a normalized Haar measure on $\mathbb{S}^1\times \mathbb{S}^1$. 
\end{enumerate}
\end{theorem}
\begin{proof}
Since $\displaystyle\inf_{P\in \mathscr{S}} M(P) = 1$, there is a sequence of pairwise distinct polynomials $(P_\nu)_\nu$ in $\mathscr{S}$ with
$H_\nu=\log M(P_\nu)\to0$ as $\nu\to \infty$. First note that this sequence has infinitely many elements of $\mathscr{S}$. If possible, assume that  it has only finitely many elements.  Then as each one is not cyclotomic, we would get $M(P_\nu)> 1$ and so is the infimum, a contradiction.

Now, we prove that $d_\nu \to \infty$ as $\nu\to \infty$. Suppose instead that for some integer $D$ there are infinitely many indices with $d_\nu\leq D$, and pass to that subsequence. If $P_\nu(x) = x^{d_\nu}+a_{d_\nu-1}x^{d_\nu-1}+\cdots+a_1x+1$ is an element of this subsequence, then 
for any $j \in \{1,2,\ldots, d_\nu-1\}$, we see that 
$$|a_j| \leq {{d_\nu} \choose j} M(P_\nu)^j \leq 2^{d_\nu}M(P_\nu)^{d_\nu}  \leq (2M(P_\nu))^{D}.$$
Since $M(P_\nu) \to 1$, we can assume that $M(P_\nu) \leq 2$ for all $\nu$ and hence we get 
$$
|a_j| \leq 4^D$$
for all $j$. Since $a_j\in \mathbb{Z}$,  only finitely many integer coefficient vectors satisfying the above inequality and hence we can produce only finitely many polynomials.  Hence only finitely many distinct polynomials can occur in this subsequence. Since every such polynomial has Mahler measure greater than $1$, this contradicts $M(P_\nu)\to1$. Thus, $d_\nu \to\infty$ as $\nu\to\infty$. This proves the first part of the assertion. 

\bigskip

For each $\nu\geq 1$, let $\alpha_\nu$ be a root of $P_\nu$ in the above sequence such that $|\alpha_\nu| > 1$. If $\beta$ is any other Galois conjugate of $\alpha_\nu$, then we see that 
$$
\frac{1}{M(P_\nu)} \leq |\beta| \leq M(P_\nu) \iff e^{-H_\nu}\leq |\beta| \leq e^{H_\nu},
$$
using the reciprocal nature of $P_\nu$.  Since $H_\nu\to 0$ as $\nu\to\infty$,  we obtain a compact annulus in $\mathbb{C}$ which contains all the Galois conjugates of $\alpha_\nu$ for all large enough $\nu$'s. Since $P_\nu$ is monic with integer coefficients, we see that the logarithmic absolute Weil height of $\alpha_\nu$ is $H_\nu/d_\nu$ which tends to $0$ as $\nu\to\infty$. Fix $0<\rho_0<\rho_1<1<\rho_2<\rho_3$. Since $H_\nu\to0$, for every sufficiently large $\nu$ all conjugates of
$\alpha_\nu$ belong to $\{\rho_1\le|z|\le \rho_2\}$ and by Proposition \ref{pro2.10000}, we get a continuous function $f$ on $\mathbb{C}$ such that $f$   is $1$ on the annulus $\rho_1\leq |z|\leq \rho_2$ and compactly supported on $\mathbb{C}$. Let $N_0$ be such that all the conjugates of $\alpha_\nu$ lie in $\{\rho_1\leq |z|\leq \rho_2\}$ for $\nu\geq N_0$, and consider the subsequence $(P_\nu)_{\nu\geq N_0}$ of $(P_\nu)_\nu$. Now, if $g$ is any continuous function on $\mathbb{C}$, then $g f$ is a continuous function with compact support and by Theorem \ref{bilu} and \eqref{eq2.10000}, we get the second assertion.

\bigskip

To prove the third assertion, let $\alpha_\nu$ be a root of $P_\nu$ and define $\xi_\nu = (\alpha_\nu, \alpha_\nu^{d_\nu})\in \overline{\mathbb{Q}}^2$. Then the logarithmic absolute Weil height 
$$
h(\alpha_\nu) = \frac{\log M(P_{\nu})}{d_\nu}=\frac{H_{\nu}}{d_{\nu}}
$$
and hence 
$$
h(\alpha_\nu^{d_\nu}) = d_\nu h(\alpha_\nu)
=\log M(P_\nu).
$$
Now, the projective height of $\xi_\nu$ is 
$$
h(1:\alpha_\nu:\alpha_\nu^{d_\nu}) = \frac{1}{[K:\mathbb{Q}]}\sum_{v\in M_K} n_v \log \max\{1, |\alpha_\nu|_v, |\alpha_\nu^{d_\nu}|_v\},  
$$
where  $K = \mathbb{Q}(\alpha_\nu)$ and $n_v$ is the local degree, if $v$ is non-archimedean, or $1$ if $v$ is a real embedding and $2$ if $v$ is a complex embedding. Since $\alpha_\nu$ is an algebraic integer, we see that 
$\log \max\{1, |\alpha_\nu|_v, |\alpha_\nu^{d_\nu}|_v\} = 0$ for all non-archimedean $v$'s. By setting $M_\infty$ as the set of all archimedean places of $K$,  we get,
$$
h(1:\alpha_\nu:\alpha_\nu^{d_\nu}) = \frac{1}{[K:\mathbb{Q}]}\sum_{v\in M_\infty} n_v \log \max\{1, |\alpha_\nu|_v, |\alpha_\nu^{d_\nu}|_v\} = d_\nu h(\alpha_\nu) = \log M(P_\nu).  
$$
Thus, the projective height of $\xi_\nu$ tends to $0$ as $\nu\to\infty$. Since $\xi_{\nu} =(\alpha_{\nu} , \alpha^{d_\nu}_{\nu})$, if $\xi_{\nu}$ belongs to the kernel of a fixed nontrivial character $z^{r}w^{s}$  then $\alpha^{r+sd_{\nu}}_{\nu}=1.$ Since $\alpha_{\nu}$ is not a root of unity , $r+sd_{\nu}=0$, which can occur only for at most one degree as $d_{\nu} \rightarrow \infty$. Therefore, by Theorem \ref{bilu}, we get the third assertion. 
\end{proof}

\begin{remark}
If one wants to solve Lehmer's problem,  one can take $\mathscr{P}$ in Theorem \ref{thm1}. By assuming the infimum is $1$,     we can get the assertion of Theorem \ref{thm1}. To get a contradiction, we need to manipulate both sides of the integrals. This is the approach we have taken for $\mathscr{P}_{m,r}$ and $\mathscr{P}_{m, a, b}$. 
\end{remark}

\section{Proof of Theorem \ref{thm9.1} and Corollary \ref{thm9.1-cor1}} \label{sec:9}

First we prove some lemmas before we prove Theorem \ref{thm9.1} and Corollary \ref{thm9.1-cor1}.

\begin{lemma}[The class $r=0$]\label{lem:zero}
Let $m\geq2$ be an integer, and let $P\in\mathbb{Z}[x]$ be  a monic, reciprocal polynomial of degree $d$, with no
cyclotomic factors in it.  If $P(x)\equiv x^d+1\pmod m$,  then
\begin{equation}\label{eq:3.1}
M(P)\geq {\sqrt{m^2+1}\over2}>1.
\end{equation}
\end{lemma}
\begin{proof}
Since $P(x)$ has no cyclotomic factors in it, we first note that both resultants
$\operatorname{Res}(P,x^d-1)$ and $\operatorname{Res}(P,x^d+1)$ are nonzero integers. 
Since $P(x) \equiv x^d+1\pmod{m}$, we see that $P(x) = (x^d+1) + m S(x)$ for some polynomial $S(x)\in \mathbb{Z}[x]$. 
Using $P(x)=(x^d+1)+mS(x)$ and the resultant identity for monic polynomials, we obtain 
$$
m^d\mid\operatorname{Res}(P,x^d+1) =\operatorname{Res}(mS(x),x^d+1).
$$
Put $D=m^2+1$ and $G(y)=(1-y)(1+y)^{m^2}.$  Consequently
\begin{equation}\label{eq:3.2}
m^{dm^2}\leq
 |\operatorname{Res}(P,x^d-1)|\,|\operatorname{Res}(P,x^d+1)|^{m^2}
 =\prod_{P(\alpha)=0}|G(\alpha^d)|.
\end{equation}
For every polynomial $H$ of degree $D$ and every $z\in\mathbb{C}$,
\begin{equation}\label{eq:3.3}
|H(z)|\leq\max_{|w|=1}|H(w)|\max\{1,|z|\}^D.
\end{equation}
Put $w=e^{it}$ in $G(w)$. Then $|G(e^{it})| = |(1-e^{it})(1+e^{it})^{m^2}| = 2^{m^2+1}|\sin(t/2)||\cos(t/2)|^{m^2}$.  By maximizing, 
  one obtains
\begin{equation}\label{eq:3.4}
\max_{|w|=1}|G(w)|
 ={2^{m^2+1}m^{m^2}\over(m^2+1)^{(m^2+1)/2}}.
\end{equation}
From the equations \eqref{eq:3.2}--\eqref{eq:3.4}, we infer that 
$$
m^{dm^2}\leq  \left({2^Dm^{m^2}\over D^{D/2}}\right)^d M(P)^{dD} =  \left({2m^{m^2/D}\over D^{1/2}}\right)^{dD} M(P)^{dD}.$$
By taking the power $1/(dD)$ both sides, we get  $M(P)\geq\sqrt D/2$.
\end{proof}

\begin{remark} 
 In Lemma  \ref{lem:zero}, put $m =2$. Then for an even integer $d\geq 2$, the polynomial $P(x) = x^d+2x^{d-1}+\cdots+2x+1$ is always reducible over $\mathbb{Q}$ and it is a product of cyclotomic polynomials and hence $M(P) = 1$. The displayed comparison polynomial is cyclotomic; this observation does not exclude other non-cyclotomic polynomials in the same residue class when $m=2$. 
\end{remark}

For an integer $m\geq 2$, we assume that $r$ satisfies $1\leq r\leq m-1$. Let
\begin{equation}\label{eq:1.1}
P(x)=x^d+a_{d-1}x^{d-1}+\cdots+a_1x+1\in\mathbb{Z}[x]
\end{equation}
be a monic, reciprocal and irreducible polynomial of degree $d\geq 2$ which is not cyclotomic.  For a root $\alpha$ of $P$,
the absolute logarithmic Weil height is
\begin{equation}\label{eq:1.3}
h(\alpha)={{\log M(P)}\over d}.
\end{equation}
Throughout the section the following assumptions are in force unless a statement explicitly says otherwise:
\begin{itemize}
\item $m\geq2$ and $1\leq r<m$ are fixed;
\item $d\geq2$ is even;
\item $P$ has the form \eqref{eq:1.1}, is reciprocal and irreducible over $\mathbb{Q}$;
\item $a_j\equiv r\pmod m$ for every $1\leq j\leq d-1$;
\item $P$ is not cyclotomic.
\end{itemize}
Choose an integer $c$ satisfying
\begin{equation}\label{eq:1.4}
c\equiv r-1\pmod m,\qquad |c|=\min_{t\in\mathbb{Z}}|r-1+tm|.
\end{equation}
When $m$ is even and $r-1\equiv m/2\pmod m$, the two minimum-length representatives are
$\pm \frac{m}{2}$, either may be chosen.  Put
\begin{equation}\label{eq:1.5}
\delta=|c|.
\end{equation}
Thus $0\leq\delta\leq m/2$.  Define the comparison polynomial
\begin{equation}\label{eq:1.6}
R_{d,c}(x)=x^d+(1+c)(x^{d-1}+\cdots+x)+1
\end{equation}
and the polynomial
\begin{equation}\label{eq:1.7}
A_{d,c}(x)=(x-1)R_{d,c}(x) =x^{d+1}+cx^d-cx-1 = x^d(x+c) - (1+cx).
\end{equation}

\begin{lemma}\label{lem:reslower}
Assume $P$, $R_{d,c}$, and $A_{d,c}$ as above such that $P\neq R_{d,c}$.  Then
\begin{equation}\label{eq:3.5}
|\operatorname{Res}(P,R_{d,c})|\geq m^d,\qquad
 |\operatorname{Res}(P,A_{d,c})|\geq m^d.
\end{equation}
\end{lemma}
\begin{proof}
Since $P(x) \ne R_{d,c}(x)$ and by hypothesis, we get   $P(x) -R_{d,c}(x)=mS(x)$
for a nonzero $S\in\mathbb{Z}[x]$ of degree less than $d$.  Since $P$ is
irreducible, $P$ and $S$ are coprime.  Therefore
$$|\operatorname{Res}(P,R_{d,c})| =m^d|\operatorname{Res}(P,S)|\geq m^d.$$
Since $d\geq2$ and $P$ is irreducible, $P(1)\ne 0$.   Since
$A_{d,c}=(x-1)R_{d,c}$, we get the   second inequality, which proves the lemma. 
\end{proof}

We shall give a lower bound for the Mahler measure for $R_{d,c}(x)$ as follows.

\begin{lemma}\label{lem:newrdc}
With the same notations as above, if $R_{d,c}(x)$ is irreducible and non-cyclotomic with $|c|\geq 2$, then $M(R_{d,c})\geq 3/2$.
\end{lemma}
\begin{proof}
If $c\geq2$, then $A_{d,c}(-c)=c^2-1>0$. Since $A_{d,c}(x)\to -\infty$ as
$x\to-\infty$,  we conclude that $R_{d,c}$ has a root whose modulus  is greater than $c$.

If $c\leq-2$ and $d\geq4$, then $A_{d,c}(3/2)<0$. For, by writing   $c=-b$ with $b\geq2$,
and $x_0=3/2$, one has
$$
A_{d,-b}(x_0)=(x_0^{d+1}-1)-b(x_0^d-x_0)  \leq(x_0^{d+1}-1)-2(x_0^d-x_0)<0,
$$
because $2(x_0^d-x_0)-(x_0^{d+1}-1)=x_0^d/2-2>0$ for $d\geq4$.
Since  $A_{d,c}(x)\to+\infty$ as $x\to+\infty$, we see that $R_{d,c}$ has a root
greater than $3/2$.  When $d=2$,  the non-cyclotomic irreducible polynomials of degree $2$ 
  have Mahler measure at least $(3+\sqrt5)/2$.  Thus in every
case with $|c|\geq2$ and $R_{d,c}$ irreducible and non-cyclotomic, we see that
$M(R_{d,c})\geq3/2$.
\end{proof}
The following estimate is the point where reciprocal nature of $P$ is important.

\begin{lemma}\label{lem:pair}
Let $|c|=\delta\geq2$, let $P$ and $A_{d,c}$ be as above, and let $H=\log M(P)$.  If $\log M(P)=H<\log(3/2)$, then
\begin{equation}\label{eq:4.1}
\log m\leq {1\over d}\log|\operatorname{Res}(P,A_{d,c})|
 \leq\left(1+{8\over d}\right)H
       +{1\over d}\sum_{P(\beta)=0}\log|\beta+c|+\log2.
\end{equation}
\end{lemma}
\begin{proof}
For any root $\beta$ of $P$, since $M(P) < 3/2$, we see that $2/3 < |\beta|< 3/2$. Since 
$$
\operatorname{Res}(P, A_{d,c}) = \prod_{P(\beta) =0} A_{d,c}(\beta),
$$
Let $\mathcal R$ contain one representative from each reciprocal pair, with $|\alpha|\geq1$ for $\alpha\in\mathcal R$. Then we get 
\begin{align}\label{thanga-new4}
\log|\operatorname{Res}(P,A_{d,c})|
&=\sum_{P(\beta)=0}\log|A_{d,c}(\beta)|
=\sum_{\alpha\in\mathcal R}
\log|A_{d,c}(\alpha)A_{d,c}(\alpha^{-1})|.
\end{align}
$
\displaystyle\sum_{\alpha\in\mathcal R}\log^+|\alpha|
=\log M(P), |\mathcal R|=d/2.$
So, we need to estimate $\log |A_{d,c}(\beta)|$. Since 
$$
A_{d,c}(\beta) = \beta^d(\beta+c) - (1+c\beta) = \beta^d(\beta+c)\left(1 - \frac{1+c\beta}{\beta^d(\beta+c)}\right).
$$ 
Let $\alpha$ be a root of $P$ such that $|\alpha|\geq 1$. Then
$$
|A_{d,c}(\alpha)|\leq |\alpha|^d|\alpha+c|\left|1-\frac{1+c\alpha}{\alpha^d(\alpha+c)}\right|.
$$
Since $|c|\geq 2$, the ratio $(1+cz)/(z+c)=b_{-1/c}(z)$ is a Blaschke factor. Since $|\alpha|\geq 1$, we conclude that $\displaystyle \left|\frac{1+c\alpha}{\alpha+c}\right|\geq 1$  and hence 
$$
\left|\frac{1+c\alpha}{\alpha^d(\alpha+c)}\right|\leq \left|\frac{1+c\alpha}{\alpha+c}\right|.
$$
Therefore, we get
$$
\left|1-\frac{1+c\alpha}{\alpha^d(\alpha+c)}\right| \leq 1+ \left|\frac{1+c\alpha}{\alpha+c}\right| \leq 2\left|\frac{1+c\alpha}{\alpha+c}\right|.
$$
Thus, for any root $\alpha$ with $|\alpha|\geq 1$ of $P$, we have 
\begin{equation*} 
\log |A_{d,c}(\alpha)| \leq d\log |\alpha|+\log |\alpha+c| + \log 2+\log\left|\frac{1+c\alpha}{\alpha+c}\right|.
\end{equation*}
Let $f(z) = \log|z+c|$ and $g(z) = \displaystyle\log \left|\frac{1+cz}{z+c}\right|$ wherever these logarithms are defined. Then $g(z) = 0$ for all $|z| = 1$. With these notations, for any root $\alpha$ with $|\alpha|\geq 1$ of $P$, we get
\begin{equation}\label{thanga-new1}
\log |A_{d,c}(\alpha)| \leq d\log |\alpha|+f(\alpha) + \log 2+ g(\alpha). 
\end{equation}
Also, note that 
$$A_{d,c}(\alpha^{-1})=-\alpha^{-(d+1)}A_{d,c}(\alpha).$$
Hence, 
$$
\log |A_{d,c}(\alpha^{-1})| = -(d+1)\log |\alpha| + \log |A_{d,c}(\alpha)|.
$$
By the definition of $g(z)$, we see that 
\begin{eqnarray}\label{thanga-new2}
f(\alpha^{-1}) &=& \log |\alpha^{-1}+c| = \log\left|\frac{1+\alpha c}{\alpha}\right|\\\nonumber
&=&-\log|\alpha| + \log|1+\alpha c| +\log |\alpha+c| - \log|\alpha+c| \\\nonumber
&=& f(\alpha) + g(\alpha) - \log |\alpha|.
\end{eqnarray}
Therefore, by \eqref{thanga-new1} and \eqref{thanga-new2}, we get
\begin{eqnarray*}
\log|A_{d,c}(\alpha)A_{d,c}(\alpha^{-1})| &= & \log |A_{d,c}(\alpha)| + \log |A_{d,c}(\alpha^{-1})|\\\nonumber
&=& 2\log |A_{d,c}(\alpha)| - (d+1)\log |\alpha|\\\nonumber
&\leq& (d-1)\log|\alpha| +2f(\alpha) + 2\log 2+2g(\alpha)\\
&=& (d-1)\log|\alpha| + f(\alpha) +2\log 2 + 2g(\alpha) + f(\alpha^{-1}) - g(\alpha) +\log|\alpha|.
\end{eqnarray*}
Hence, we get
\begin{equation}\label{thanga-new3}
\log|A_{d,c}(\alpha)A_{d,c}(\alpha^{-1})| \leq  d\log|\alpha| + f(\alpha)+f(\alpha^{-1}) + g(\alpha) +2\log 2.
\end{equation}
Now, we estimate as $g(\alpha) \leq 8\log|\alpha|$. For, consider the radial segment $z(u) = e^{u+it}$ for all $0\leq u\leq \log |\alpha|$ so that $g(z(0)) = g(e^{it}) = 0$ and $z(\log |\alpha|) = \alpha$ and hence $g(z(\log|\alpha|))= g(\alpha)$. Therefore, 
$$
g(\alpha) = g(z(\log|\alpha|)) - g(z(0)) = \int_0^{\log |\alpha|} \frac{d}{du}\bigl(g(z(u))\bigr)\,du
$$
and hence 
$$
|g(\alpha)| \leq \int_0^{\log |\alpha|}\left|\frac{d}{du}\bigl(g(z(u))\bigr)\right|\,du.
$$
Now consider
$$
\frac{d}{du}(g(z(u))) = \frac{d}{du}\log \left|\frac{1+ce^{u+it}}{e^{u+it}+c}\right| = \Re\left(\frac{cz(u)}{1+cz(u)}\right) - \Re\left(\frac{z(u)}{c+z(u)}\right).
$$
Thus, 
$$
\left|\frac{d}{du}(g(z(u)))\right| \leq \left|\frac{cz(u)}{1+cz(u)}\right| + \left|\frac{z(u)}{c+z(u)}\right| \leq \frac{\delta |z(u)|}{\delta|z(u)| -1} + \frac{|z(u)|}{\delta-|z(u)|}.
$$
Since $|z(u)|$ satisfies $2/3 < |z(u)|<3/2$, we see that $\delta |z(u)| \geq 4/3$ and hence we get
$$
\left|\frac{d}{du}(g(z(u)))\right| \leq \left(1+ \frac{1}{\delta|z(u)| -1}\right) + \frac{|z(u)|}{\delta - |z(u)|} \leq 4+3=7<8.
$$
Thus, we get $g(\alpha) \leq 8\log |\alpha|$. Therefore, by \eqref{thanga-new3}, we get
$$
\log|A_{d,c}(\alpha)A_{d,c}(\alpha^{-1})| \leq (d+8)\log |\alpha| +f(\alpha)+f(\alpha^{-1})+2\log 2.
$$
Now, we sum over the representatives $\alpha\in\mathcal R$ to get
$$
\sum_{\alpha\in\mathcal R}\log|A_{d,c}(\alpha)A_{d,c}(\alpha^{-1})| \leq (d+8)\sum_{\alpha\in\mathcal R} \log |\alpha| +\sum_{\alpha\in\mathcal R}(f(\alpha)+f(\alpha^{-1})) + d\log 2.
$$
Since $\displaystyle \sum_{|\alpha|> 1} \log |\alpha| = \log M(P) = H$, by Lemma \ref{lem:reslower} (since by Lemma \ref{lem:newrdc}, $P\ne R_{d,c}$) and by \eqref{thanga-new4},   divide by $d$ both sides, we get the assertion.    Note that the roots on the unit circle cause no issues as they can be  paired with their inverses, and above equation holds with $|\alpha|=1$.
\end{proof}

\begin{proposition}\label{prop:qual}
Assume $\delta\geq2$ and $m>2\delta$.   For a fixed $r\mod{m}$, let 
$\mathscr{P}_{m,r}$ be the collection of all monic, reciprocal and irreducible polynomials $P(x)\in \mathbb{Z}[x]$ of degree $d$ for any integer $d\geq 2$ such that its inner coefficients are congruent to $r\pmod{m}$ and non-cyclotomic. Then 
$$
\inf_{P\in\mathscr{P}_{m,r}} M(P) > 1.
$$
\end{proposition}

\begin{proof}
Suppose $\displaystyle\inf_{P\in\mathscr{P}_{m,r}} M(P) =1$. Then by Theorem \ref{thm1} (2), 
there is a sequence of pairwise distinct polynomials $P_\nu\in \mathscr{P}_{m,r}$ with
$H_\nu=\log M(P_\nu)\to0$.  Also note that, by Lemma \ref{lem:newrdc}, we cannot have  $P_\nu=R_{d_\nu,c}$ for infinitely many $\nu$, as $M(P_\nu)\to 1$. 
Then by Theorem \ref{bilu},   we conclude that their Galois conjugates
equi-distribute on $\mathbb{S}^1$ and we invoke Theorem \ref{thm1} (2).    Choose an annular cutoff $\chi$ from Proposition~\ref{pro2.10000}, supported away from $-c$ and equal to $1$ on an annulus containing all conjugates for sufficiently large $\nu$. Apply Theorem~\ref{thm1}(2) to the compactly supported continuous function $\varphi(z)=\chi(z)\log|z+c|$.   Since $|c|=\delta\geq2$ and by Jensen's
formula, we can get 
\begin{equation}\label{eq:5.1}
\lim_{\nu\to\infty}{1\over d_\nu}
 \sum_{P_\nu(\beta)=0}\log|\beta+c|
 ={1\over2\pi}\int_0^{2\pi}\log|e^{it}+c|\,dt
 =\log\delta.
\end{equation}
Since $M(P_\nu)\to 1$, we can assume that $M(P_\nu)< 3/2$ for all $\nu$.  Then by Lemma \ref{lem:pair}, we get 
$$
\log m \leq\lim_{\nu\to\infty} \left(1+\frac{8}{d_\nu}\right) H_\nu +\lim_{\nu\to\infty} \frac{1}{d_\nu}\sum_{P_\nu(\beta) = 0}\log|\beta+c| +\log 2
$$
which implies 
$$
\log m\leq \lim_{\nu\to\infty} \frac{1}{d_\nu}\sum_{P_\nu(\beta) = 0}\log|\beta+c| +\log 2  = \log \delta+\log 2,
$$
by \eqref{eq:5.1}. Thus, we get  $\log m  \leq\log(2\delta)$, a contradiction to $m>2\delta$.  
\end{proof}

\begin{proposition}[Explicit generic bound]\label{prop:effective}
Under the hypotheses of Corollary \ref{thm9.1-cor1}, for $r\geq 3$, $\delta \geq 2$ and $m> 2\delta$, we have  \begin{equation}\label{eq:6.8}
M(P)\geq\exp(\eta(m,\delta)).
\end{equation}
\end{proposition}
\begin{proof}
Write $H=\log M(P)$.  If $H\geq\log(3/2)$, there is nothing to prove.
Assume $H<\log(3/2)$.   By Lemma \ref{lem:newrdc},   we may assume
$P\neq R_{d,c}$.

\smallskip

For $z\in\mathbb{C}^*$, define the real-valued function
$$F_c(z)=\log\left|c+{z\over|z|}\right|. 
$$
This function depends only on the direction of $z$ and is constant along rays. By Jensen's formula, the circle average of $F_c$ is $\log\delta$.   By letting $f(\theta) = \log |c+e^{2\pi i\theta}|$, we see that 
$$
f'(\theta) = \operatorname{Re}\left(\frac{2\pi ie^{2\pi i\theta}}{c+e^{2\pi i\theta}}\right).
$$
Therefore, 
$$
|f'(\theta)| \leq \frac{2\pi}{|c+e^{2\pi i\theta}|} \leq \frac{2\pi}{|c| - 1} = \frac{2\pi}{\delta-1}. 
$$
Indeed, for $u,v\in\mathbb S^1$, $|\log|c+u|-\log|c+v||\leq |u-v|/(\delta-1)$; since $|u-v|\leq2\pi d_{\mathrm{lp}}(u,v)$ and $F_c$ is constant along rays, $\operatorname{Lip}(F_c)\leq2\pi/(\delta-1)$. Now by writing
$$
\log|c+e^{2\pi i\theta}| = \log |c| + \log \left|1+ \frac{e^{2\pi i\theta}}{c}\right|,
$$
we get the absolutely convergent Fourier expansion of $F_c$ on the circle, which is 
$$F_c(e^{2\pi i\theta})=\log\delta+
 \operatorname{Re}\sum_{k\geq1}{(-1)^{k+1}e^{2\pi ik\theta}
                                  \over k c^k}.$$
Thus, for the logarithmic-polar metric in Theorem \ref{sombra}, we get
$$\operatorname{Lip}(F_c)\leq{2\pi\over\delta-1},\qquad
 \|\widehat{F_c'}\|_1={2\pi\over\delta-1}.$$
Therefore \eqref{eq:2.2} permits 
$$
c(F_c)\leq 2 \operatorname{Lip}(F_c) + 16 \|\widehat{F_c'}\|_1 \leq \frac{4\pi}{\delta-1} + \frac{32\pi}{\delta-1} = \frac{36\pi}{\delta-1} =  K_\delta. 
$$ Applying \eqref{eq:2.1}, using
$h(\alpha)=H/d$, we get 
\begin{equation}\label{eq:6.9}
\left|{1\over d}\sum_{P(\beta)=0}F_c(\beta)-\log\delta\right|
 \leq K_\delta
 \sqrt{{4H+64\log(d+1)\over d}}.
\end{equation}
On the annulus $2/3<|z|<3/2$,  by differentiating, we get 
$$|\log|z+c|-F_c(z)|\leq R_\delta|\log|z||
$$
where $\displaystyle R_\delta={3\over2\delta-3}$. Since every root $\beta$ of $P$ satisfies $2/3 < |\beta| < 3/2$, we see that 
$$
|\log|\beta+c| - F_c(\beta)| \leq R_\delta|\log|\beta||.
$$
The reciprocal nature of the roots implies that 
\begin{equation}\label{eq:6.10}
\sum_{P(\beta)=0}|\log|\beta||=2H.
\end{equation}
Combining Lemmas \ref{lem:reslower} and \ref{lem:pair} with
\eqref{eq:6.9}--\eqref{eq:6.10}, and subtracting $\log(2\delta)$, yields
\begin{equation}\label{eq:6.11}
\log\left(\frac{m}{2\delta}\right) = \Delta\leq
 \left(1+{8+2R_\delta\over d}\right)H
 +K_\delta\sqrt{{4H+64\log(d+1)\over d}}
\end{equation}
Since $R_\delta\leq3$ and $d\geq2$, the coefficient of $H$ in
\eqref{eq:6.11} is less than $15$.

If $d\geq D_0$, equations \eqref{eq:6.5} and \eqref{eq:6.11}, together with
$\sqrt{x+y}\leq\sqrt x+\sqrt y$, give
\begin{equation}\label{eq:6.12}
{3\Delta\over4}\leq15H+2K_\delta\sqrt{H/d}.
\end{equation}
If $H\geq\Delta/40$, this is one of the terms in \eqref{eq:6.7}.  Otherwise
$15H<3\Delta/8$, and \eqref{eq:6.12} gives
\begin{equation}\label{eq:6.13}
H\geq {9\Delta^2d\over256K_\delta^2}
       \geq{9\Delta^2\over128K_\delta^2}.
\end{equation}
If $4\leq d<D_0$, Theorem \ref{voutier} implies $H>V(D_0)$.  Finally, we have 
$d=2$, and $P=x^2+ax+1$. Then the  irreducibility and non-cyclotomic nature of $P$ forces
that $|a|\geq3$, and hence $H\geq\log((3+\sqrt5)/2)$.  These cases prove
\eqref{eq:6.8}.  
\end{proof}

\bigskip

\noindent{\bf Proof of Corollary \ref{thm9.1-cor1}.}  When $r=0$, we see that $P\equiv x^d+1\pmod m$, and by Lemma \ref{lem:zero}, 
we get the result.  When $r=1$, all coefficients of $P$ are $1$ modulo $m$,  and by Theorem \ref{borwein}, we get the required bound.

If $r=2$, then note that 
$$R_{d,1}(x)=(x+1)(1+x+\cdots+x^{d-1}).$$
Thus $1+x+\cdots+x^{d-1}$ divides $P$ in $(\mathbb{Z}/m\mathbb{Z})[x]$. Take $D= n = d$ in Theorem \ref{sliver}, along with $\log m \geq \log 2 > 1/16 = D/(16n)$ for every $m\geq 2$. Therefore, by Lemma 2.1 in  \cite{silver} , we get the third line of \eqref{eq:6.14}.

\smallskip

If $m$ is even and $\delta=m/2$, put $q=m/2$.  Since
$r\equiv1\pmod q$, note that every coefficient of $P$ is congruent to $1$ modulo $q$.  Here $q\geq2$ unless $m=2$, and the case $m=2$ was already covered by $r=0,1$. Thus, this bound follows from Theorem \ref{borwein}. All remaining cases, that is, $\delta\geq2$ and $2\delta<m$ are covered in  Proposition \ref{prop:effective}. Hence the corollary follows. \hfill$\square$

\bigskip

\noindent{\bf Proof of Theorem \ref{thm9.1}}. The classes $r=0$ and $r=1$ follow from Lemma 3.1 and Theorem 2.4. The class $r=2$ is treated by Silverman's bound as used in Corollary 1.1. If $m$ is even and $\delta=m/2$, reduction modulo $m/2$ gives the corresponding Borwein--Dobrowolski--
Mossinghoff bound. All remaining $r$ satisfy $\delta\ge2$ and $2\delta<m$ and follow from Proposition~\ref{prop:effective}. \hfill$\square$

\bigskip
\section{Proof of Theorem \ref{thm:2,3,-1-residues}: the case $r=2$}\label{sec:3}

We shall start with the following lemma. 
\begin{lemma}\label{lem:Res-esti}
Let $m, d\ge 2$ be integers. Let $P(x)=x^d+a_{d-1}x^{d-1}+\cdots+a_1x+1\in \mathbb{Z}[x]$ be a polynomial which 
satisfies $a_j\equiv2\pmod m$ for $1\leq j\leq d-1$. Suppose that $P$ has no cyclotomic factors. Then, for every integer $j\geq1$,
$$
\left|\mathop{\rm Res}(P,x^{jd}-1)\right|\geq m^{d-1}.
$$
\end{lemma}

\begin{proof}
By the hypothesis,  modulo $m$, we have
$$
P(x)\equiv x^d+2x^{d-1}+\cdots+2x+1.
$$
The polynomial on the right factors as
$$
x^d+2x^{d-1}+\cdots+2x+1=(x+1)G_d(x)
$$
where 
$$
G_d(x)=1+x+\cdots+x^{d-1}=\frac{x^d-1}{x-1}.
$$
Thus there exists $B(x)\in\mathbb Z[x]$ such that
$$P(x)=(x+1)G_d(x)+mB(x).
$$
Since the resultant is unchanged, when a multiple of the second polynomial is subtracted from the first,  we get
$$
\left|\mathop{\rm Res}(P,G_d)\right|=\left|\mathop{\rm Res}(mB,G_d)\right|.
$$
Since $\deg G_d=d-1$,
$$\left|\mathop{\rm Res}(P,G_d)\right|=m^{d-1}\left|\mathop{\rm Res}(B,G_d)\right|.
$$
Since every root of  $G_d$ is a root of unity and $P$ has no cyclotomic factors, we get 
$$\gcd(P,G_d)=1.$$
Hence
$$
\mathop{\rm Res}(P,G_d)\neq0.
$$
Therefore, it follows that
$$\left|\mathop{\rm Res}(P,G_d)\right|\geq m^{d-1}.
$$
Since $x^d-1=(x-1)G_d(x)$ and by the 
multiplicative nature  of the resultant,  we get 
$$
\mathop{\rm Res}(P,x^d-1)=\mathop{\rm Res}(P,x-1)\mathop{\rm Res}(P,G_d).
$$
Both factors are nonzero integers, because $P$ has no root of unity. Therefore
$$\left|\mathop{\rm Res}(P,x^d-1)\right|\geq m^{d-1}.
$$
Finally, for any $j$, we see that:
$$
x^{jd}-1=(x^d-1)(1+x^d+\cdots+x^{(j-1)d}).
$$
Hence
$$
\mathop{\rm Res}(P,x^{jd}-1)=\mathop{\rm Res}(P,x^d-1)\mathop{\rm Res}(P,1+x^d+\cdots+x^{(j-1)d}).
$$
The second factor is again a nonzero integer, since all roots of
$1+x^d+\cdots+x^{(j-1)d}$ are roots of unity. Thus we get 
$$
\left|\mathop{\rm Res}(P,x^{jd}-1)\right|\geq m^{d-1}.
$$
\end{proof}

\begin{theorem}\label{thm:2(m)}
Let $m\geq2$ and $d\geq 2$ be  integers, and let $P(x)=x^d+a_{d-1}x^{d-1}+\cdots+a_1x+1\in\mathbb Z[x]$ 
be a monic reciprocal polynomial such that $a_j\equiv2\pmod m$ for every $1\leq j\leq d-1$, and suppose that $P$ has no cyclotomic factors. Then, for every integer $J\geq1$,
$$\log M(P)\geq\frac{3}{J+2}\left(\frac{d-1}{d}\log m-\frac{\log(J/2+1)+1}{J}\right).$$
\end{theorem}

\begin{proof}
Let $\displaystyle P(x)=\prod_{\nu=1}^d(x-\alpha_\nu)$ be the factorization over $\mathbb{C}$. 
Since $P$ is monic, we have
$$
\mathop{\rm Res}(P,x^{jd}-1)=\prod_{\nu=1}^d(\alpha_\nu^{jd}-1).
$$
Then by Lemma  \ref{lem:Res-esti}, we get
\begin{equation}\label{Thm3.1-eq1}
(d-1)\log m\leq\sum_{\nu=1}^d\log|\alpha_\nu^{jd}-1|.
\end{equation}
For every root $\alpha_\nu$, we define
$$
z_\nu=
\begin{cases}
\alpha_\nu^d,&|\alpha_\nu|\leq1,\\ \alpha_\nu^{-d},&|\alpha_\nu|>1.
\end{cases}
$$
Then $|z_\nu|\leq1$. If $|\alpha_\nu|\leq1$, then
$$\log|\alpha_\nu^{jd}-1|=\log|1-z_\nu^j|.$$
If $|\alpha_\nu|>1$, then
$$\log|\alpha_\nu^{jd}-1|=jd\log|\alpha_\nu|+\log|1-z_\nu^j|.$$
By summing over all roots, we get 
$$\sum_{\nu=1}^d\log|\alpha_\nu^{jd}-1|=jd\log M(P)+\sum_{\nu=1}^d\log|1-z_\nu^j|.$$
Consequently, by \eqref{Thm3.1-eq1},  for every $j\geq1$,
$$(d-1)\log m\leq jd\log M(P)+\sum_{\nu=1}^d\log|1-z_\nu^j|.
$$
Fix an integer $J\geq1$  and put
$$w_j=1-\frac{j}{J+1}.
$$
Multiplying $w_j$ with the preceding inequality  and then summing over $1\leq j\leq J$, we obtain
$$
\frac {J}2(d-1)\log m\leq d\log M(P)\sum_{j=1}^Jjw_j+\sum_{\nu=1}^d\sum_{j=1}^Jw_j\log|1-z_\nu^j|.
$$
By applying the definition of $w_j$ and when we compute, we get 
$$\sum_{j=1}^Jjw_j=\frac{J(J+2)}6.
$$
By Lemma \ref{lem:complex}, we get
$$
\sum_{j=1}^Jw_j\log|1-z_\nu^j|\leq\frac12\log\left(\frac J2+1\right)+\frac12.
$$
By combining both, we get
$$\frac {J}2(d-1)\log m\leq\frac{dJ(J+2)}6\log M(P)+\frac d2\left(\log\left(\frac J2+1\right)+1\right).$$
Multiplying by $6/(dJ(J+2))$, we obtain
$$\log M(P)\geq\frac{3}{J+2}\left(\frac{d-1}{d}\log m-\frac{\log(J/2+1)+1}{J}\right),
$$
as desired.
\end{proof}

\begin{corollary}\label{cor1:thm3.1}
Under the hypotheses of Theorem \ref{thm:2(m)}, we have
$$
M(P)\geq\exp\left(\frac{3}{19}\left(\frac{\log2}{2}-\frac{\log(19/2)+1}{17}\right)\right) \geq 1.02482\ldots.
$$
 \end{corollary}

\begin{proof}
Since $P$ is a reciprocal polynomial which has no cyclotomic factors, the degree $d\geq 2$ and must be an even integer. 
Hence we have $\displaystyle\frac{d-1}{d}\geq\frac12.$ Since $m\geq2$, $\log m\geq\log2.$
By taking  $J=17$ in Theorem \ref{thm:2(m)}, we get 
$$\log M(P)\geq\frac{3}{19}\left(\frac{\log2}{2}-\frac{\log(19/2)+1}{17}\right),$$
which proves the stated bound.
\end{proof}

\begin{corollary}\label{cor2:thm3.1}
Under the hypotheses of Theorem \ref{thm:2(m)}, we have
$$M(P)\geq\left(\frac m4\right)^{d/(d+1)}.$$
Consequently, if $m\geq5$, then we get 
$$M(P)\geq\left(\frac m4\right)^{2/3}.$$
\end{corollary}

\begin{proof}
Since $P(x)\equiv(x+1)G_d(x)\pmod m,$
we see that:
$$(1-x)P(x)\equiv(1+x)(1-x^d):=H_d(x)\pmod m.$$
Thus there exists $S(x)\in\mathbb Z[x]$ such that
$$H_d(x)=(1-x)P(x)+mS(x).$$
Then  $\mathop{\rm Res}(P,H_d)=m^d\mathop{\rm Res}(P,S).$ 
Since $H_d$ is a product of cyclotomic polynomials and $P$ has no cyclotomic factor, the resultant is nonzero. Therefore
$$|\mathop{\rm Res}(P,H_d)|\geq m^d.
$$
For each root $\alpha$ of $P$, we see that 
$$|H_d(\alpha)|\leq\|H_d\|_{\infty}\max\{1,|\alpha|\}^{d+1},
$$
and hence
$$
m^d\leq |\mathop{\rm Res}(P,H_d)| = \prod_{P(\alpha) = 0}|H_d(\alpha)| \leq \|H_d\|_\infty^{d} M(P)^{d+1}.
$$
 Note that 
$$
\|H_d\|_{\infty}=\max_{|z|=1}|1+z||1-z^d|\leq4.
$$
Hence $m^d\leq4^dM(P)^{d+1}$ and therefore, we get the desired bound. 
When $m\geq5$, we see that $m/4>1$. Since $d\geq2$,$\displaystyle\frac{d}{d+1}\geq\frac23,$ 
and the second assertion follows.
\end{proof}

\section{Proof of Theorem \ref{thm:2,3,-1-residues}: the case $r=3$}\label{sec:4}

\subsection{Coefficients are $3\pmod 5$} We begin with the following lemma.

\begin{lemma}\label{lem1:thm4.1} 
Let $d\geq 2$ be an even integer. If $A_d(z) = 1+2z-2z^d-z^{d+1}$ and $B_d(z) = 2-z+z^d-2z^{d+1}$ are in $\mathbb{C}[z]$, then 
$$\|A_dB_d\|_{\infty}\leq\frac{25}{2}.$$
\end{lemma}

\begin{proof}
Put
$$
z=e^{it},\qquad u=z^d,\qquad c=\cos t,\qquad s=\sin t.
$$
By multiplying both the polynomials, we get
$$
A_d(z)B_d(z)=2+3z-2z^2-3(1+z^2)z^d+(2z^2+3z-2)z^{2d} = 2+3z-2z^2-3(1+z^2)u+(2z^2+3z-2)u^{2}.
$$
Dividing by $z$, and putting $z = e^{it}$,  we get  
$$
z^{-1}A_d(z)B_d(z)=(3-4is)-6cu+(3+4is)u^2.
$$
After division by $u$, we notice that the first and third terms are complex conjugates. Therefore
$$
|A_d(z)B_d(z)|\leq2\sqrt{9+16s^2}+6|c|.
$$
Since $9+16s^2=25-16c^2,$ by putting $X=\sqrt{25-16c^2},\qquad Y=4|c|,$
 we get $X^2+Y^2=25,$ and $\displaystyle 2X+6|c|=2X+\frac32Y. $ Now, by the Cauchy--Schwarz inequality,
$$
2X+\frac32Y\leq\sqrt{4+\frac94}\sqrt{X^2+Y^2}=\frac52\cdot5=\frac{25}{2},
$$
as desired. 
\end{proof}

\begin{theorem}\label{thm4.1} 
For an even integer $d\geq 2$, let $P(x) = x^d+a_{d-1}x^{d-1}+\cdots+a_1x+1\in \mathbb{Z}[x]$ be a monic irreducible and reciprocal polynomial which has no cyclotomic factors. Suppose $a_j \equiv 3\pmod{5}$ for all $1\leq j\leq d-1$. Then $M(P) \geq 2^{3/7}$. Moreover, if $d\geq 6$, then $M(P) \geq 2^{d/(2d+2)}.$
\end{theorem}

\begin{proof}
Given that 
$$P(x)\equiv R_{3,d}(x):=1+3x+3x^2+\cdots+3x^{d-1}+x^d\pmod{5}.
$$
Let 
$$
A_d(x)=1+2x-2x^d-x^{d+1} \mbox{ and } B_d(x)=2-x+x^d-2x^{d+1}.
$$
Since $(1-x)R_{3,d}(x)=A_d(x)$,  we have $A_d(x)\equiv(1-x)P(x)\pmod5$.  Since $B_d(x)\equiv2A_d(x)\pmod5$, we see that  $B_d(x)\equiv2(1-x)P(x)\pmod5$.  Therefore 
$$5^d\mid\mathop{ \rm Res}(P,A_d) \mbox{ and } 5^d\mid\mathop{ \rm Res}(P,B_d).
$$
\noindent{\bf Claim 1.}  We have $\mathop{ \rm Res}(P,B_d) \ne 0.$

\smallskip
If possible, $P$ divides $B_d$. Since $\deg B_d=d+1$, the quotient must be linear. By comparing the constant and the leading coefficients, we get  
$$
B_d(x)=2(1-x)P(x).
$$
The coefficient of $x$ on the right hand side of the polynomial is $2a_1-2,$ whereas the coefficient of $x$ in $B_d$ is $-1$. Thus
$2a_1-2=-1,$ which is impossible for an integer $a_1$, which proves Claim 1. 

\smallskip

If $\mathop{ \rm Res}(P,A_d)=0$, then $P$ divides $A_d$. Then by comparing the constant and the leading coefficients, we get  
$ A_d(x)=(1-x)P(x)$ and hence we conclude that $P(x) = R_{3,d}(x) $. Since $d$ is even, we see that 
$$A_d(-2)=1-4-2^{d+1}+2^{d+1}=-3 \mbox{ and  } A_d(-2) = 3R_{3,d}(-2) \implies R_{3,d}(-2) = -1.
$$
Since $R_{3,d}(x)\to\infty$ as $x\to-\infty$,   it has a real zero $\leq -2$ which implies $M(R_{3,d}) = M(P) \geq 2$, as claimed. Hence we can assume that $\mathop{ \rm Res}(P,A_d)\ne 0$. 

Since $\mathop{ \rm Res}(P,B_d) \ne 0$ and $\mathop{ \rm Res}(P,A_d)\ne 0$, we see that 
$$
5^{2d} \leq |\mathop{ \rm Res}(P,A_dB_d)|.
$$
Since the degree of $A_dB_d$ is $2d+2$, by Lemma \ref{lem:resultant-mahler} and Lemma \ref{lem1:thm4.1}, we get 
$$5^{2d}\leq\left(\frac{25}{2}\right)^dM(P)^{2d+2}
$$
which gives the degree-dependent estimate. 

When $d\geq 6$, we see that $\displaystyle \frac{d}{2d+2} \geq \frac 37$. We treat the cases when $d=2$ and $d = 4$ separately as follows.

\smallskip

If $d=2$, then  $P(x)=x^2+ax+1$ where $a\equiv3\pmod5$.   If $|a|\leq2$, the only possibility is $a=-2$, and hence we get 
$P(x)=(x-1)^2$, a contradiction as $P$ has no cyclotomic factors.  If $|a|\geq3$,
$$
M(P)\geq\frac{3+\sqrt5}{2}>2^{3/7}
$$
as desired. 

Now let $d=4$. Write $P(x)=x^4+ax^3+bx^2+ax+1$ where $a, b\equiv 3\pmod 5$. 

Suppose, for a contradiction, that $M(P)<2^{3/7}<\frac32.$ Then the coefficients have to satisfy   
$$
|a|<6,\qquad |b|<9.
$$
Since $a,b\equiv3\pmod5$, we must have $a\in\{-2,3\}$ and $b\in\{-7,-2,3,8\}$.  By putting $y=x+x^{-1}$, we get
$$
P(x)=x^2Q(y) \mbox{ where }  Q(y)=y^2+ay+b-2.
$$
When $(a,b)=(-2,3)$ and $(a,b)=(3,-2)$, we see that 
$$
x^4-2x^3+3x^2-2x+1=(x^2-x+1)^2 \mbox{ and } x^4+3x^3-2x^2+3x+1=(x^2-x+1)(x^2+4x+1)
$$
and in both cases cyclotomic factors are present. Therefore these cases are not possible. 

For each remaining case with $b\neq8$, the quadratic $Q$ has a real root $\beta$ satisfying
$$|\beta|>\frac52.$$
Indeed, the largest absolute values among the roots are at least
$$1+\sqrt5,\qquad1+\sqrt{10},\qquad\frac{3+\sqrt5}{2},\qquad\frac{3+3\sqrt5}{2},$$
according to the relevant pair $(a,b)$.

\smallskip

Note that if $|\beta|\geq5/2$ and $\beta=\lambda+\lambda^{-1}$, 
then one of the two roots $\lambda$ has modulus at least $2$ which in turn implies $M(P)\geq2$. Hence we can assume that $b = 8$.  

\smallskip

If $b=8$, the roots of $Q$ are non-real and  conjugates with product $6$, and hence each has modulus $\sqrt6$. Let $r>1$ be the modulus of the root of $x^2-\beta x+1$ outside the unit circle. Then
$$
\sqrt6=|\beta|\leq r+r^{-1} \implies  r\geq\frac{\sqrt6+\sqrt2}{2},
$$
and the conjugate value of $\beta$ contributes another root of modulus $r$. Hence
$$
M(P)\geq r^2\geq2+\sqrt3>2.
$$
This completes the proof of the theorem.
\end{proof}

\subsection{Coefficients are $3\pmod{m}$ for $m\geq 7$} We prove the following result. 

\begin{theorem}\label{thm4.2} 
For any integer $m\geq 7$, and for any even integer $d\geq 2$, let $P(x) = x^d+a_{d-1}x^{d-1}+\cdots+a_1x+1\in \mathbb{Z}[x]$ be a monic irreducible and reciprocal polynomial which has no cyclotomic factors. Suppose $a_j \equiv 3\pmod{m}$ for all $1\leq j\leq d-1$. Then
$$
M(P) \geq \min\left\{2, \left(\frac{m}{6}\right)^{d/(d+1)}\right\} \geq \min\left\{2, \left(\frac{m}{6}\right)^{2/3}\right\}.
$$
\end{theorem}
\begin{proof}
As before, we see that $(1-x)P(x)\equiv A_d(x)\pmod m,$ where
$A_d(x)=1+2x-2x^d-x^{d+1}.$

\smallskip

If $\mathop{ \rm Res}(P,A_d)=0$, then comparison of leading and constant coefficients implies that 
$A_d=(1-x)P$,  and therefore we get $P=R_{3,d}$. We have already proved that this polynomial has a real root smaller than $-2$ and hence $M(P) \geq 2$.

We can assume that the resultant is nonzero. Then
$$
m^d\leq|\mathop{ \rm Res}(P,A_d)|.
$$
Note that  the Blaschke factorization
$$
A_d(z)=(2+z)\left(\frac{1+2z}{2+z}-z^d\right)
$$
shows that
$$\|A_d\|_{\infty}\leq6.
$$
Since $\deg A_d=d+1$, by Lemma  \ref{lem:resultant-mahler}, we get 
$$
m^d\leq6^dM(P)^{d+1},
$$
which proves the theorem. 
\end{proof}

\section{Proof of Theorem \ref{thm:2,3,-1-residues}: the case $r = m-1$}\label{sec:5}

\subsection{Assume that $m = 5$} We shall start with the following lemma. 

\begin{lemma}\label{lem1:thm5.1} For an even integer $d\geq 2$, let 
$$
 C_d(z)=1-2z+2z^d-z^{d+1} \mbox{  and }  D_d(z)=2+z-z^d-2z^{d+1}$$
 be polynomials of degree $d+1$ over $\mathbb{C}$. Then 
$$
\|C_dD_d\|_{\infty}\leq\frac{25}{2}.
$$
\end{lemma}
\begin{proof} 
The proof is analogous to that of Lemma \ref{lem1:thm4.1} and we give a sketch of the proof here. 
Write $C_d(z)=(1-2z)+z^d(2-z)$ and $D_d(z)=(2+z)-z^d(1+2z)$. Multiplying out,
$$
C_d(z)D_d(z)=(2-3z-2z^2)+3(1+z^2)z^d-(2+3z-2z^2)z^{2d}.
$$
With $z=e^{it}$, $u=z^d$, $c=\cos t$ and $s=\sin t$, this gives
$$
z^{-1}C_d(z)D_d(z)=-\Bigl[(3+4is)-6cu+(3-4is)u^2\Bigr],
$$
which is the expression in the proof of Lemma \ref{lem1:thm4.1} with $s$ replaced by $-s$. The rest of that proof applies verbatim and gives $\|C_dD_d\|_\infty\le25/2$.
\end{proof}

\begin{theorem}\label{thm5.1}
For an even integer $d\geq 2$, let $P(x) = x^d+a_{d-1}x^{d-1}+\cdots+a_1x+1\in \mathbb{Z}[x]$ be a monic irreducible and reciprocal polynomial which has no cyclotomic factors. Suppose $a_j \equiv -1\pmod{5}$ for all $1\leq j\leq d-1$. Then $M(P) \geq 2^{3/7}$. Moreover, if $d\geq 6$, then $M(P) \geq 2^{d/(2d+2)} \geq 2^{3/7}.$
\end{theorem}

\begin{proof}
By reducing modulo $5$, we get 
$$P(x)\equiv R_{-1,d}(x):=1-x-x^2-\cdots-x^{d-1}+x^d \pmod{5}.
$$
Define
$$
C_d(x)=1-2x+2x^d-x^{d+1} \mbox{  and } D_d(x)=2+x-x^d-2x^{d+1}.
$$
Then, we get
$$
C_d(x)=(1-x)R_{-1,d}(x) \mbox{ and hence } C_d(x)\equiv(1-x)P(x)\pmod5.
$$
Also, we see that  $D_d(x)\equiv2C_d(x)\pmod5$.  Therefore
$$
5^d\mid\mathop{ \rm Res}(P,C_d) \mbox{ and }  5^d\mid\mathop{ \rm Res}(P,D_d).
$$
We claim that $\mathop{ \rm Res}(P,D_d)\ne 0$. For if $P$ divides $D_d$, then by comparing  the leading and the constant coefficients, we get $D_d(x)=2(1-x)P(x).$ The coefficient of $x$ on the right hand side is $2a_1-2,$ whereas the coefficient of $x$ in $D_d$ is $1$. Thus we get $2a_1-2=1,$ which is impossible.

\smallskip

If $\mathop{ \rm Res}(P,C_d) = 0$, then $P$ divides $C_d$. By comparing the leading and the constant coefficients, we get 
$C_d(x)=(1-x)P(x),$  and hence $P=R_{-1,d}$.

When $d=2$, we see that $R_{-1,2}(x) = x^2 - x+1 = \Phi_6$, and hence $P \ne R_{-1,2}$, by hypothesis. For even integer $d\geq 4$, we see that $(x-1)R_{-1,d}(x)=x^d(x-2)+2x-1.$  When we plug $x=\sqrt2$, we get, 
$$
x^d(x-2)+2x-1\leq4(\sqrt2-2)+2\sqrt2-1=6\sqrt2-9<0.
$$
Since $R_{-1,d}(x)\to \infty$ as $x\to \infty$, it has a real root larger than $\sqrt2$. Hence $M(P) = M(R_{-1,d}) \geq \sqrt2$, as desired. Therefore, we can assume that $\mathop{ \rm Res}(P,C_d) \ne 0$. 

Since both the resultants are nonzero, by Lemma \ref{lem:resultant-mahler},  we get 
$$
5^{2d}\leq\left(\frac{25}{2}\right)^dM(P)^{2d+2},
$$
which gives the degree-dependent estimate. 

If $d=2$, then $P(x)=x^2+ax+1$ where  $a\equiv-1\pmod5.$ The value $a=-1$ gives $\Phi_6$, cyclotomic polynomial. Every other admissible value satisfies $|a|\geq4$, and hence
$$M(P)\geq2+\sqrt3>2^{3/7}.$$

Now let $P(x)=x^4+ax^3+bx^2+ax+1.$ Suppose, for a contradiction, that $M(P)<2^{3/7}<\frac32.$
Then $|a|<6,\qquad |b|<9,$ and hence  $a\in\{-1,4\}$ and  $b\in\{-6,-1,4\}.$
As in the proof of  Theorem \ref{thm4.1}, we let  $P(x)=x^2Q(x+x^{-1}),$
where $Q(y)=y^2+ay+b-2.$

For every pair $(a, b)$,  except  when $(a,b)=(-1,4)$, the polynomial $Q$ has a real root $\beta$ satisfying
$\displaystyle |\beta|>\frac{13}{6}.$ Indeed, the relevant largest absolute roots are
$$\frac{1+\sqrt{33}}2,\qquad\frac{1+\sqrt{13}}2,\qquad2+2\sqrt3,\qquad2+\sqrt7,\qquad2+\sqrt2.
$$
Since
$$|\beta|>\frac{13}{6}=\frac32+\frac23,
$$
then the equation $\lambda+\lambda^{-1}=\beta$  has a root with modulus greater than $3/2$ and thus  we get $M(P)>3/2$.

When $(a,b)=(-1,4)$, we have  $Q(y)=y^2-y+2,$ whose roots are
$$
\beta=\frac{1\pm i\sqrt7}{2}.
$$
Write the root outside the unit circle of
$$x^2-\beta x+1$$
as
$$\lambda=re^{i\theta},\qquad r>1.$$
Then we get
$$
\left|\mathop{\rm Im}\beta\right|=|(r-r^{-1})\sin\theta|\leq r-r^{-1}.
$$
Therefore
$$
r-r^{-1}\geq\frac{\sqrt7}{2},
$$
and hence
$$
r\geq\frac{\sqrt7+\sqrt{23}}4>\frac32.
$$
The conjugate value of $\beta$ contributes another root of modulus $r$. Thus, $M(P)\geq r^2>3/2$.
This proves the theorem.
\end{proof}

\subsection{Assume $m =3$} We start with the following lemma.

\begin{lemma}\label{lem2:thm5.2} For an even integer $d\geq 2$, let 
$$
 E_d(z)= 1+z-z^d - z^{d+1}, F_d(z) = 2 - z +z^d - 2z^{d+1}, \mbox{  and }  G_d(z)=-1+2z-2z^d+z^{d+1}$$
 be polynomials of degree $d+1$ over $\mathbb{C}$. Then 
$$
\|E_dF_dG_d\|_{\infty}\leq 12\sqrt{3}.
$$
\end{lemma}
\begin{proof} 
We let
$$
z=e^{it},\qquad u=z^d=e^{i\phi},\qquad c=\cos t.
$$
Also, by letting $a=2-z,\qquad b=2z-1$, we see that
$$
F_d(z)=a-ub,  G_d(z)=b-ua, \mbox{  and } E_d(z)=(1+z)(1-u).
$$
On the unit circle, $b=z\overline a$ and $|a|^2=5-4c.$
Write
$$a=\varrho e^{i\theta},\qquad\varrho^2=5-4c,  \mbox{ and put }
q=\frac{\phi}{2},\qquad\delta=\frac{t-2\theta}{2}.
$$
A direct calculation implies
$$
|F_d(z)G_d(z)|=4\varrho^2|\sin(q+\delta)\sin(q-\delta)|.
$$
Also,
$$|E_d(z)|=4\sqrt{\frac{1+c}{2}}|\sin q|.
$$
By  setting $v=|\sin q|$  and  $h=\sin^2\delta$, we get
$$
|E_d(z)F_d(z)G_d(z)|=16\varrho^2\sqrt{\frac{1+c}{2}}v|v^2-h|.
$$
For $0\leq v,h\leq1$, we get
$$
v|v^2-h|\leq\max\left\{1-h,\frac{2h^{3/2}}{3\sqrt3}\right\}.
$$
For, indeed, on the interval $v^2\geq h$, the function $v(v^2-h)$ is increasing and has maximum $1-h$. On the interval $v^2\leq h$, the function $v(h-v^2)$ has its maximum at $v=\sqrt{h/3}$. Now
$$
\cos(2\delta)=\mathop{\rm Re}\left(\frac{b}{a}\right)=\frac{5c-4}{5-4c},
$$
and hence  $\displaystyle h=\frac{9(1-c)}{2(5-4c)}.$

For the first branch,
$$
|E_d(z)F_d(z)G_d(z)| \leq 16(5-4c)\sqrt{\frac{1+c}{2}}(1-h)=16\left(\frac{1+c}{2}\right)^{3/2}\leq16.
$$
For the second branch, the corresponding quantity is
$$
|E_d(z)F_d(z)G_d(z)| \leq 24\sqrt3\sqrt{\frac{(1-c)^3(1+c)}{5-4c}}.
$$
The identity
$$
5-4c-4(1-c)^3(1+c)=(2c^2-2c-1)^2 \implies \frac{(1-c)^3(1+c)}{5-4c}\leq\frac14.
$$
Thus the second branch is at most $12\sqrt3.$  Since $16<12\sqrt3$, the result follows.  
\end{proof}

\begin{theorem}\label{thm5.2}
For an even integer $d\geq 2$, let $P(x) = x^d+a_{d-1}x^{d-1}+\cdots+a_1x+1\in \mathbb{Z}[x]$ be a monic irreducible and reciprocal polynomial which has no cyclotomic factors. Suppose $a_j \equiv -1\pmod{3}$ for all $1\leq j\leq d-1$. Then 
$$
M(P) \geq \left(\frac{3\sqrt3}{4}\right)^{d/(3d+3)} \geq \left(\frac{3\sqrt3}{4}\right)^{2/9}>1.$$ 
\end{theorem}
\begin{proof}
If $P=R_{-1,d}$ and $d\geq4$, then $M(P)>\sqrt2$, as in the proof of Theorem \ref{thm5.1}. The case $d=2$ is $\Phi_6$, which is  excluded, as $P$ has no cyclotomic factors.

Now we prove that  all three resultants
$$
\mathop{ \rm Res}(P,E_d) \ne 0,\qquad\mathop{ \rm Res}(P,F_d)\ne 0,\qquad\mathop{ \rm Res}(P,G_d)\ne 0
$$
where $E_d, F_d$, and $G_d$ are as defined in Lemma \ref{lem2:thm5.2}.  

Since $E_d$ is a product of cyclotomic polynomials,  $P$  cannot divide it. If $P$ divides $F_d$, by comparing the  constant and the leading coefficients, we get   $F_d=2(1-x)P.$ Then, by comparing  the coefficient of $x$, we get 
$2a_1-2=-1,$ which is impossible. Finally, if $P$ divides $G_d$, then by comparing the leading and the constant coefficients, we get
$G_d=(x-1)P.$ Therefore $P=R_{-1,d},$  and in this case, we have already proved in Theorem \ref{thm5.1} that $M(P) \geq \sqrt{2}$. 
Thus, we can assume that all three resultants are nonzero. Each is divisible by $3^d$, and hence 
$$
3^{3d}\leq|\mathop{ \rm Res}(P,E_dF_dG_d)|.
$$
The product has degree $3d+3$, and its unit-circle norm is at most $12\sqrt3$ by Lemma \ref{lem2:thm5.2}. Therefore
$$3^{3d}\leq(12\sqrt3)^dM(P)^{3d+3}.
$$
Since $\displaystyle\frac{27}{12\sqrt3}=\frac{3\sqrt3}{4},$ we get the assertion. 
\end{proof}

\subsection{Assume $m\geq 7$}  We prove the following.

\begin{theorem}\label{thm5.3}

Let $m\ge 7$, and let $d\ge 2$ be even. Suppose
$P(x)=x^d+a_{d-1}x^{d-1}+\cdots+a_1x+1\in\mathbb Z[x]$
is monic irreducible and reciprocal, has no cyclotomic factors,
and $a_j\equiv-1\pmod m$ for $1\le j\le d-1$.
Then $$M(P)\ge\min\left\{\rho,
\left(\frac m6\right)^{d/(d+1)}\right\}
\ge\min\left\{\rho,\left(\frac m6\right)^{2/3}\right\},$$
where $\rho$ is the largest real root of
$x^4-x^3-x^2-x+1$.
\end{theorem}
\begin{proof}
Consider the auxiliary polynomial $C_d(x)=1-2x+2x^d-x^{d+1}.$ Then 
we have $C_d(x)\equiv(1-x)P(x)\pmod m.$ 

If the resultant vanishes, then $C_d=(1-x)P$,  and hence  $P=R_{-1,d}$.  
In this case, when $d = 2$, $R_{-1,d}$ is cyclotomic and hence we assume that $d\geq 4$. Note that the largest real root of $R_{-1,4} = x^4 - x^3-x^2-x+1$, say, $\rho$ which lies in $(1, 2)$ and $\rho \geq 1.72208\ldots$. For all $1< x < 2$ and any $d\geq 4$, we have $x^d(x-2) \leq x^4(x-2)$. Hence we get
$$
(x-1)R_{-1, d}(x) \leq (x-1)R_{-1,4}(x).
$$
Hence $R_{-1,d}(x)$ has a real root at least $\rho$. Therefore, $M(P) = M(R_{-1,d}) \geq \rho$ for all $d\geq 4.$
 Thus, we can assume that $\mathop{ \rm Res}(P,C_d) \ne 0$. Therefore, we get 
$$m^d\leq|\mathop{ \rm Res}(P,C_d)|.
$$
The Blaschke factorization leads to 
$$C_d(z)=(2-z)\left(z^d+\frac{1-2z}{2-z}\right)
$$
and the identity $|1-2z|=|2-z|$ on $|z|=1$ imply that
$$\|C_d\|_{\infty}\leq6.
$$
Therefore
$$
m^d\leq6^dM(P)^{d+1},$$
which proves the asserted lower bound.
\end{proof}

\section{The case $m = 5, a=4, b=2$ of Proof of Theorem \ref{thm:abmodm}}\label{sec:7}

We shall start with the following lemma.

\begin{lemma}\label{lem1:thm8.1}
Let $d\geq 2$ be an even integer. Let $A_d(z) = (1-z+z^2)(1-z^d)$ and $B_d(z) = (-3-2z+2z^2) - z^d(2-2z-3z^2)$ be complex polynomials of degree $d+2$. Then
$$
\|A_dB_d\|_\infty \leq 25 \frac{975}{976}<25.
$$
\end{lemma}

\begin{proof}
Put
$$
U(z) = 1-z+z^2, \quad V(z) = -3-2z+2z^2, \quad V^*(z) = z^2V(z^{-1}) = 2-2z-3z^2
$$
so that $A_d(z) = U(z)(1-z^d)$ and $B_d(z) = V(z) - z^dV^*(z)$.
We let
$$
z=e^{it},\qquad u=z^d,\qquad c=\cos t,\qquad s=\sin t.
$$
Then, on the unit circle, we see that $z^{-1}U(z)=2c-1.$ Also, $z^{-1}V(z)=-(c+2)+5is.$
 
 Now write $w=-(c+2)+5is.$  Since $\displaystyle z^{-1}V^*(z)=\overline w,$ we have
$$
A_d(z)=z(2c-1)(1-u) \mbox{ and }  B_d(z)=z(w-u\overline w).
$$
By writing $u=e^{i\phi}$ and $w=\eta e^{i\theta}$, we see that 
$$
|1-u|=2\left|\sin\frac{\phi}{2}\right| \mbox{ and } 
|w-u\overline w|=2\eta\left|\sin\left(\theta-\frac{\phi}{2}\right)\right|.
$$
Since  
$2\sin a \  \sin(\theta-a)=\cos(2a-\theta)-\cos\theta$ is true, 
we get
$$
|\sin a\sin(\theta-a)|\leq\frac{1+|\cos\theta|}{2}.
$$
Therefore
$$
|A_d(z)B_d(z)|\leq2|2c-1|\eta(1+|\cos\theta|).
$$
Now
$$
\eta^2=|w|^2=(c+2)^2+25s^2=29+4c-24c^2, \mbox{ and } |\cos\theta|=\frac{c+2}{\eta},
$$
because $c+2>0$. Consequently,
$$|A_d(z)B_d(z)|\leq F(c),
$$
where $\displaystyle F(c)=2|2c-1|\left(\sqrt{29+4c-24c^2}+c+2\right).$
\bigskip

Now, we first prove that $F(c)<25$ for every $-1\leq c\leq1$. Put 
$$t=|2c-1|  \mbox{ and }  Y=2t\sqrt{29+4c-24c^2}.
$$
Then we see that  $F(c)=2t(c+2)+Y.$

If $c\geq1/2$, from the direct expansion and since $96(2c\mp1)^2 = 24(4c\mp2)^2$, we get 
\begin{eqnarray*}
\left(25-2(2c-1)(c+2)\right)^2-4(2c-1)^2(29+4c-24c^2) &=&25(4c^2-2c-5)^2+(4c-2)^2+96(2c-1)^2\\
&=&25\left[(4c^2-2c-5)^2 +(4c-2)^2\right],
\end{eqnarray*}
If $c\leq1/2$, from the direct expansion  and since $96(2c\mp1)^2 = 24(4c\mp2)^2$, we get
\begin{eqnarray*}
\left(25-2(1-2c)(c+2)\right)^2-4(1-2c)^2(29+4c-24c^2)&=&25(4c^2-2c-3)^2+(4c+2)^2+96(2c+1)^2\\
&=&25\left[(4c^2-2c-3)^2+(4c+2)^2\right].
\end{eqnarray*}
Note that the expressions on the right are positive. Moreover, $25-2t(c+2)>0$
 and hence  $F(c)<25$.

\bigskip

We now record an explicit uniform bound. On $1/2\leq c\leq1$, $|4c^2-2c-5|\geq3,$
and hence  the corresponding sum of squares is at least $9$.

On $-1\leq c\leq1/2$, we claim that
$$
(4c^2-2c-3)^2+(4c+2)^2\geq\frac1{16}.
$$
If $|4c+2|\geq1/4$, then nothing to prove. If $|4c+2|<1/4$, then
$$-\frac9{16}<c<-\frac7{16}.
$$
On this interval,  we have $4c^2-2c-3\leq-\frac{39}{64},$
and hence its square is greater than $1/16$. Thus in both ranges the relevant sum of squares is at least $1/16$. Thus in both ranges the sum of squares in brackets is at least $1/16$, and hence
$$
(25-F(c))\,\bigl(25-2t(c+2)+Y\bigr)=\bigl(25-2t(c+2)\bigr)^2-Y^2\ge\frac{25}{16}.
$$
Since $t\leq3$  and $29+4c-24c^2<36,$ we have $Y<36.$  Therefore $25-2t(c+2)+Y\leq61.$
Consequently,
$$
25-F(c)\geq\frac{25}{16\cdot61}=\frac{25}{976}.
$$
Hence
$$F(c)\leq25-\frac{25}{976}=25\frac{975}{976}.
$$
This proves
$$\|A_dB_d\|_{\infty}\leq25\frac{975}{976}.$$
\end{proof}

Now we prove the main result of this section.

\begin{theorem}\label{thm8.1}
Let $P(x) = x^d +a_{d-1}x^{d-1}+\cdots+a_1x+1\in \mathbb{Z}[x]$ be a monic irreducible reciprocal polynomial of even degree $d\geq 2$ which has no cyclotomic factors in it. Let $a_j \equiv 2\pmod{5}$ for all even indices $1\leq j\leq d-1$ and $a_i\equiv -1\pmod{5}$ for all odd indices $1\leq i\leq d-1$. Then 
$$
M(P)\geq \left(\frac{976}{975}\right)^{1/3} > 1
$$
and for all $d\geq 4$, we have 
$$
 M(P)\geq \left(\frac{976}{975}\right)^{d/(2d+4)}.
$$
\end{theorem}

\begin{proof} Define $U(x)=1-x+x^2$ and the residue polynomial is
$$R_{\rm alt}(x)=1-x+2x^2-x^3+2x^4-\cdots-x^{d-1}+x^d.
$$
It factors as
$$R_{\rm alt}(x)=U(x)(1+x^2+x^4+\cdots+x^{d-2}).
$$
Indeed, multiplication by $U$ gives the coefficient $-1$ at every odd position, the coefficient $2$ at every even interior position, and coefficient $1$ at the two endpoints.  Consequently,
$$(1-x^2)R_{\rm alt}(x)=U(x)(1-x^d).$$
By defining $A_d(x)=U(x)(1-x^d)$, we get $A_d(x)\equiv(1-x^2)P(x)\pmod5.$
We require a second lift. Put 
$$V(x)=-3-2x+2x^2
\mbox{ and } V^*(x)=x^2V(x^{-1})=2-2x-3x^2.
$$
By reducing modulo $5$, we get 
$$V(x)\equiv2U(x) 
\mbox{ and } V^*(x)\equiv2U(x).
$$
By defining $B_d(x)=V(x)-x^dV^*(x)$, we see that 
$$
B_d(x)\equiv2U(x)(1-x^d)\equiv2(1-x^2)P(x)\pmod5.
$$
Thus
$$5^d\mid\mathop{ \rm Res}(P,A_d) \mbox{ and }  5^d\mid\mathop{ \rm Res}(P,B_d).
$$
We prove that both the resultants are non-zero. 

When $d\geq 4$, we first claim that $\mathop{ \rm Res}(P,B_d)\ne 0$.  Suppose that $P$ divides $B_d$. Since $\deg B_d=d+2$, write $B_d(x)=Q(x)P(x),$ where $Q(x)=3x^2+bx-3.$

The constant coefficient of $Q$ is $-3$ because the constant coefficient of $B_d$ is $-3$, and the leading coefficient is $3$ because the leading coefficient of $B_d$ is $3$. The coefficient of $x$ in $B_d$ is $-2$, and hence 
$-3a_1+b=-2.$  The coefficient of $x^{d+1}$ in $B_d$ is $2$. By reciprocity, $a_{d-1}=a_1,$ and by comparing
 the coefficient of $x^{d+1}$, we get $3a_1+b=2.$

By subtracting the two equations, we get $6a_1=4,$ which is impossible.

For $d\geq 4$, the polynomial $P$ is non-cyclotomic, and since $A_d(x) = U(x)(1-x^d)$ is the product of cyclotomic polynomials, we conclude that $\mathop{ \rm Res}(P,A_d)\neq0$. Thus both resultants are nonzero.

Suppose first that $d\geq4$. Both auxiliary resultants are nonzero and divisible by $5^d$. Hence
$$
5^{2d}\leq|\mathop{ \rm Res}(P,A_dB_d)|.
$$
Since $\deg(A_dB_d)=2d+4$ and by Lemma \ref{lem1:thm8.1}, 
$$\|A_dB_d\|_{\infty}\leq25\frac{975}{976},
$$
we obtain
$$25^d\leq\left(25\frac{975}{976}\right)^dM(P)^{2d+4}.$$
Thus
$$
M(P)^{2d+4}\geq\left(\frac{976}{975}\right)^d,
$$
and therefore
$$M(P)\geq\left(\frac{976}{975}\right)^{d/(2d+4)}.
$$
Also, for $d\geq4$, we see that $\displaystyle \frac{d}{2d+4}\geq\frac13,$ and hence we get
$$
M(P)\geq\left(\frac{976}{975}\right)^{1/3}.
$$

When $d=2$, the condition concerns only the coefficient of $x$, which is congruent to $-1$ modulo $5$. Thus
$P(x)=x^2+ax+1$ where $a\equiv-1\pmod5.$ 

When the value $a=-1$, we note that $P(x) = \Phi_6(x)$, sixth cyclotomic polynomial, which is not the case as $P$ is not cyclotomic.  Every other admissible value satisfies $|a|\geq4$, and hence
$$
M(P)\geq 2+\sqrt3>\left(\frac{976}{975}\right)^{1/3}.
$$
This proves the theorem.
\end{proof}

\section{Proof of Theorem \ref{thm:abmodm}}

First, when $d=1$, every monic, reciprocal, irreducible, non-cyclotomic quadratic has the form $P(x)=x^2+ux+1$ with $|u|\geq3$, and therefore $M(P)\geq\frac{3+\sqrt5}{2}>1$. Now, let $m\geq 3$ be an integer and $0\leq a,b\leq m-1$ be distinct integers. For each integer $d\geq2$, let $\mathscr{P}_{m,a,b}(d)$ be the collection of all monic, reciprocal, and irreducible non-cyclotomic polynomials $P(x) = x^{2d} + c_{2d-1}x^{2d-1}+\cdots+c_2x^2+c_1x+1\in \mathbb{Z}[x]$ such that $c_j = c_{2d-j}$, $c_{2j+1} \equiv a\pmod{m}$ for $0\leq j\leq d-1$ and $c_{2j}\equiv b\pmod{m}$ for $1\leq j\leq d-1$. The family $\mathscr{P}_{m,a,b}=\bigcup_{d\geq1}\mathscr{P}_{m,a,b}(d)$ also includes the degree-two case treated above.

\bigskip

Choose integers $A,E$ satisfying $A\equiv a\pmod m$ and $E\equiv b-1\pmod m.$ Define 
$$
G_{A,E}(x)=x^2+Ax+E \mbox{ and } G_{A,E}^*(x)=x^2G_{A,E}(1/x)=1+Ax+Ex^2.
$$
For a positive integer $d$, define 
$$
Q_{d,A,E}(x)=x^{2d}+1+A\sum_{j=0}^{d-1}x^{2j+1}+(E+1)\sum_{j=1}^{d-1}x^{2j}.
$$ 
The odd interior coefficients of $Q_{d,A,E}$ are equal to $A$, while the even interior coefficients are equal to $E+1$. Since $A\equiv a\pmod m$ and $E+1\equiv b\pmod m,$ every coefficient of $P(x)-Q_{d,A,E}(x)$ is divisible by $m$. Therefore there exists a polynomial 
$$
H_{A,E}(x)\in\mathbb{Z}[x] \mbox{ such that }  P(x)=Q_{d,A,E}(x)+mH_{A,E}(x).
$$ 
Since $P$ and $Q_{d,A,E}$ have the same leading and the constant coefficients,  we have $\deg H_{A,E}<2d.$

\bigskip 
We now compute a sparse multiple of $Q_{d,A,E}$ as follows. The geometric-series identities imply
$$
(x^2-1)\sum_{j=0}^{d-1}x^{2j+1}=x^{2d+1}-x  \mbox{ and } (x^2-1)\sum_{j=1}^{d-1}x^{2j}=x^{2d}-x^2.
$$ 
Consequently, 
$$
(x^2-1)Q_{d,A,E}(x)=x^{2d+2}+Ax^{2d+1}+Ex^{2d} -Ex^2-Ax-1
$$ 
and therefore 
\begin{equation}\label{eq1000}
(x^2-1)Q_{d,A,E}(x)=x^{2d}G_{A,E}(x)-G_{A,E}^*(x).
\end{equation}
Now we need the following lemma.

\begin{lemma}\label{thmab:lem1}
For each $P\in \mathscr{P}_{m,a,b}$ of degree $2d\geq 4$, there exists a pair $(A, E)\in \mathbb{Z}^2$ such that $A\equiv a\pmod{m}, E\equiv b-1\pmod{m}$, the Mahler measure $M(x^2 +Ax+E) <m$ and if $P(x) = Q_{d, A, E}(x) + mH_{A,E}(x)$, then $H_{A,E}\ne 0$.
\end{lemma}
\begin{proof}
We shall construct a finite set $\mathcal{S}$ (note that $|\mathcal{S}| \leq 2$) of   integer pairs $(A,E)$ such that  $A\equiv a\pmod m$,   $E\equiv b-1\pmod m$ along with the Mahler measure  $M(x^2+Ax+E)<m$  for every $(A,E)\in\mathcal{S}$.

We first observe the following.  For every monic quadratic polynomial $G(x)=x^2+Ax+E,$  by Jensen's inequality and Parseval's identity,  we get 
$$ 
\log M(G)=\frac{1}{2\pi}\int_0^{2\pi}\log|G(e^{it})|\,dt \leq\frac12\log\left(\frac{1}{2\pi}\int_0^{2\pi}|G(e^{it})|^2\,dt\right)  =\frac12\log(1+A^2+E^2).
$$ 
Therefore $M(G)\leq\sqrt{1+A^2+E^2}$ for every monic quadratic $G(x) = x^2 +Ax+E$. 

\bigskip

Assume first that $m\geq3$. Choose centered representatives $A_0,E_0\in\mathbb{Z}$ such that 
$$A_0\equiv a\pmod m, \quad E_0\equiv b-1\pmod m, \quad |A_0|\leq\frac m2 \mbox{ and } |E_0|\leq\frac m2.
$$ 
Therefore 
$$1+A_0^2+E_0^2\leq1+\frac{m^2}{2}<m^2, \mbox{ and hence } M(x^2+A_0x+E_0)<m.
$$
If $(A_0,E_0)=(0,0)$, we take $\mathcal{S}=\{(0,0)\}.$ If $(A_0,E_0)\neq(0,0)$, put 
$$u=\max\{|A_0|,|E_0|\} \mbox{ and } v=\min\{|A_0|,|E_0|\}.
$$ 
Then $$1\leq u\leq\frac m2 \mbox{ and } 0\leq v\leq u.
$$
If $|A_0|\geq|E_0|$, define 
$$
A_1=A_0-m\operatorname{sgn}(A_0) \mbox{ and } E_1=E_0.
$$ 
If $|E_0|>|A_0|$, define $A_1=A_0$ and $E_1=E_0-m\operatorname{sgn}(E_0).$ In either case, 
$$
A_1\equiv a\pmod m \mbox{ and } E_1\equiv b-1\pmod m.
$$ 
The shifted coordinate has absolute value $m-u$, while the other coordinate has absolute value $v$. Thus we get
 $$1+A_1^2+E_1^2=1+(m-u)^2+v^2.
 $$ 
 Moreover, 
 $$m^2-\left(1+(m-u)^2+v^2\right)=2mu-u^2-v^2-1\geq4u^2-u^2-v^2-1\geq2u^2-1 >0.
 $$ 
 Here we used $m\geq2u$ and $v\leq u$. Hence the Mahler measure $M(x^2+A_1x+E_1)<m.$ 
 In this case define $\mathcal{S}=\{(A_0,E_0),(A_1,E_1)\}.$

\bigskip

For completeness, we also record the case $m=2$, although Theorem~\ref{thm:abmodm} assumes $m\geq3$. Since $a\not\equiv b\pmod2$, there are two cases. 

If $a\equiv0\pmod2$ and $b\equiv1\pmod2,$ define $\mathcal{S}=\{(0,0)\}.$ 

If $a\equiv1\pmod2$ and $b\equiv0\pmod2,$ define $\mathcal{S}=\{(-1,-1),(1,-1)\}.$ 

For each pair in the second set, we have $1+A^2+E^2=3<4$ and the corresponding quadratic has Mahler measure less than $2$.

We have therefore constructed a fixed finite set $\mathcal{S}$, depending only on $m,a,b$, such that the Mahler measure $M(x^2+Ax+E)<m$ for every $(A,E)\in\mathcal{S}$. The set contains either two distinct pairs or the single pair $(0,0)$.

\bigskip

Let $P\in\mathscr{P}_{m,a,b}$ have degree $2d\geq4$. For each pair $(A,E)\in\mathcal{S}$, write 
$$
P=Q_{d,A,E}+mH_{A,E}.
$$
If $\mathcal{S}$ contains two distinct pairs, at least one of the two polynomials $H_{A,E}$ is nonzero. Indeed, if both were zero, then two distinct comparison polynomials would be equal. Subtracting them would give 
$$
\left(A_1-A_2\right)\sum_{j=0}^{d-1}x^{2j+1}+\left(E_1-E_2\right)\sum_{j=1}^{d-1}x^{2j}=0.
$$ 
The first sum contains only odd powers and the second contains only even powers. Since $d\geq2$, both sums are nonempty, and hence we get $A_1=A_2 \mbox{ and }E_1=E_2,$ a contradiction.

\smallskip

If $\mathcal{S}=\{(0,0)\}$, then 
$$
Q_{d,0,0}(x)=1+x^2+x^4+\cdots+x^{2d}=\frac{x^{2d+2}-1}{x^2-1}.
$$ 
Hence $Q_{d,0,0}$ is a product of cyclotomic polynomials. Since $P$ is irreducible and non-cyclotomic, 
$P\neq Q_{d,0,0},$ and therefore $H_{0,0}\neq0.$ 

\bigskip

In conclusion,  for every $P\in \mathscr{P}_{m,a,b}$ of degree at least $4$, one can choose a pair $(A,E)\in\mathcal{S}$ such that $H_{A,E}\neq0$ and the Mahler measure $M(x^2+Ax+E)<m.$
\end{proof}

\bigskip

\noindent{\bf Proof of Theorem \ref{thm:abmodm}}. Suppose  that 
$$\inf_{P\in \mathscr{P}_{m,a,b}}M(P)=1.
$$ 
Then by Theorem \ref{thm1} (1) applied to $S = \mathscr{P}_{m,a,b}$, there exists a sequence of pairwise distinct polynomials $P_k\in\mathscr{P}_{m,a,b}$ such that 
$$M(P_k)\longrightarrow1 \mbox{ as } k\to\infty.
$$ 
Write $d_k=\deg P_k.$ By Theorem \ref{thm1} (1), we see that $d_k\to\infty$ as $k\to\infty$.  In particular, $d_k\geq4$ for every sufficiently large $k$. For each such $P_k$, by Lemma \ref{thmab:lem1}, choose a pair $(A,E)\in\mathcal{S}$ for which the corresponding remainder is nonzero. Since $\mathcal{S}$ is finite, after passing to another subsequence we may use one fixed pair $(A,E)$ for all $k$. Put 
$$
G(x)=x^2+Ax+E \mbox{ and }G^*(x)=1+Ax+Ex^2.
$$ 
Then by Lemma \ref{thmab:lem1}, $M(G)<m.$

\bigskip

Choose one root $\alpha_k$ of $P_k$ and define $\boldsymbol{\xi}_k=(\alpha_k,\alpha_k^{d_k}).$  Let $h$ denote the logarithmic absolute Weil height. Since $P_k$ is monic and irreducible of degree $d_k$, 
$$
h(\alpha_k)=\frac{\log M(P_k)}{d_k}.
$$ 
Also, 
$$h(\alpha_k^{d_k})=d_kh(\alpha_k)=\log M(P_k).
$$ 
For the pair $(\alpha_k, \alpha_k^{d_k})$, the projective height satisfies  
$$
h(1:\alpha_k:\alpha_k^{d_k})\leq h(\alpha_k)+h(\alpha_k^{d_k}).
$$ 
Therefore, as $k\to\infty$, 
$$h(\boldsymbol{\xi}_k)\leq\left(1+\frac{1}{d_k}\right)\log M(P_k)\longrightarrow0.$$

We next verify strictness of the sequence $(\boldsymbol{\xi}_k)_k$. Suppose that this sequence is not strict. We know that every proper algebraic subgroup of $\mathbb{G}_m^2$ is contained in the kernel of a nontrivial character $$(z,w)\longmapsto z^rw^s$$ for some $(r,s)\in\mathbb{Z}^2\setminus\{(0,0)\}$. 

If  $\boldsymbol{\xi}_k$  belonged to one fixed proper algebraic subgroup for infinitely many $k$, then for fixed integers $r,s$, not all zero, such that $$\alpha_k^{r+sd_k}=1$$ for infinitely many $k$. Since $P_k$ is irreducible and non-cyclotomic, $\alpha_k$ is not a root of unity. Hence $$r+sd_k=0.$$ For fixed $r,s$, this can hold for at most one value of $d_k$, whereas $d_k\to\infty$, a contradiction. Therefore the sequence $(\boldsymbol{\xi}_k)_k$ is strict.

 \bigskip

Define the two-variable polynomial 
$$
F(z,w)=wG(z)-G^*(z).
$$ 
For each $k$, write $P_k=Q_{d_k/2,A,E}+mH_k$, with $H_k\neq0$, since $\deg P_k=d_k$. 

At a root $\alpha$ of $P_k$, the sparse identity \eqref{eq1000} implies  
\begin{equation}\label{eq1001}
F(\alpha,\alpha^{d_k})=(\alpha^2-1)Q_{d_k/2,A,E}(\alpha)=-m(\alpha^2-1)H_k(\alpha).
\end{equation}
Note that the polynomial $P_k$ has no common root with $(x^2-1)H_k$. Indeed, $P_k$ is irreducible, $H_k\neq0$, and $\deg H_k<\deg P_k,$  and hence $P_k$ does not divide $H_k$. Also, $P_k$ has neither $1$ nor $-1$ as a root, because it is irreducible and non-cyclotomic. Hence we get 
$$
\left|\operatorname{Res}\bigl(P_k,(x^2-1)H_k\bigr)\right|\geq1.
$$ 
By multiplying over all $d_k$ conjugates of $\alpha_k$ and by \eqref{eq1001}, we get 
\begin{eqnarray}\label{eq1003}
\frac{1}{d_k}\sum_{\sigma}\log\left|F\left(\sigma(\alpha_k),\sigma(\alpha_k)^{d_k}\right)\right| &=& \frac{1}{d_k}\sum_\sigma\log|\sigma(\alpha_k)^2-1| + \log m + \frac{1}{d_k}\sum_{\sigma}\log |H_k(\sigma(\alpha_k))| \nonumber\\
&= &\log m +\frac{1}{d_k}\log\left|\operatorname{Res}\bigl(P_k,(x^2-1)H_k\bigr)\right| \nonumber\\
&\geq&\log m,
\end{eqnarray}
 where the sum runs over all embeddings  $\sigma: \mathbb{Q}(\alpha_k)\to \mathbb{C}$.

By setting $L=2(1+|A|+|E|),$  we see that for all $z,w\in\mathbb{C}^\times$, the triangle inequality implies 
\begin{equation}\label{eq1002}
|F(z,w)|\leq L\max\{1,|z|\}^2\max\{1,|w|\}.
\end{equation}
Define $$\Psi(z,w)=\log|F(z,w)|-\log L-2\log^+|z|-\log^+|w|,$$ where $$\log^+x=\max\{0,\log x\}.$$ At a zero of $F$, we interpret $\Psi$ as $-\infty$. By \eqref{eq1002}, we see that $\Psi\leq0.$

\bigskip

Let $\mu_k$ be the uniform probability measure on the Galois orbit $S_k$ of $\boldsymbol{\xi}_k$. Since 
$$
\frac{1}{d_k}\sum_{\sigma}\log^+|\sigma(\alpha_k)|=\frac{\log M(P_k)}{d_k}\mbox{  and }\frac{1}{d_k}\sum_{\sigma}\log^+|\sigma(\alpha_k)^{d_k}|=\log M(P_k),
$$ 
the resultant inequality in \eqref{eq1003} implies that  
\begin{eqnarray*}
\int_{S_k}\Psi\,d\mu_k &=&\frac{1}{d_k}\sum_{\xi\in S_k}\Psi(\xi) = \frac{1}{d_k}\sum_{(\alpha, \alpha^{d_k})\in S_k}\log |F(\alpha, \alpha^{d_k})| - \log L\\
&&\qquad -\frac{2}{d_k}\sum_{(\alpha, \alpha^{d_k})\in S_k}\log^+|\alpha| - \frac{1}{d_k}\sum_{(\alpha, \alpha^{d_k})\in S_k}\log^+|\alpha|^{d_k}\\ 
 &\geq&\log m-\log L-\left(1+\frac{2}{d_k}\right)\log M(P_k).
\end{eqnarray*}
Therefore 
\begin{equation}\label{eq9999}
\liminf_{k\to\infty}\int_{S_k}\Psi\,d\mu_k\geq\log m-\log L.
\end{equation}
For $T>0$, we define 
$$\Psi_T(z,w)=\max\{\Psi(z,w),-T\}.
$$ 
Then $\Psi_T$ is bounded and continuous on $(\mathbb{C}^\times)^2$.
Since $\Psi\leq\Psi_T,$ we have 
$$
\int_{S_k}\Psi\,d\mu_k\leq\int_{S_k}\Psi_T\,d\mu_k.
$$ 
Since $(\xi_k)$ is strict and $h(\xi_k)\to0$, Theorem~\ref{bilu} (with $N=2$) applies to it. Writing $H_k=\log M(P_k)$ for this paragraph, reciprocity gives
$$
\int\bigl(|\log|z||+|\log|w||\bigr)\,d\mu_k(z,w)=\frac{2H_k}{d_k}+2H_k\longrightarrow0,
$$
so the measures $\mu_k$ concentrate in any fixed compact annular neighbourhood of $(\mathbb{S}^1)^2$. Multiplying $\Psi_T$ by a compactly supported continuous cutoff equal to $1$ on such a neighbourhood and applying Theorem~\ref{bilu} yields the convergence for $\Psi_T$. 
Hence 
$$
\lim_{k\to\infty}\int_{S_k}\Psi_T\,d\mu_k=\int_{(\mathbb{S}^1)^2}\Psi_T\,d\lambda,
$$ 
where $\lambda$ is normalized Haar measure on $(\mathbb{S}^1)^2$. Hence, by \eqref{eq9999}, we get  
$$
\log m-\log L\leq\int_{(\mathbb{S}^1)^2}\Psi_T\,d\lambda
$$ 
for every $T>0$.

As $T\to\infty$, the functions $\Psi_T$ decrease pointwise to $\Psi$. Note that the logarithm of the absolute value of a nonzero Laurent polynomial is integrable on the unit torus. Since $0\leq-\Psi_T\uparrow-\Psi$, the monotone convergence theorem gives 
$$
\lim_{T\to\infty}\int_{(\mathbb{S}^1)^2}\Psi_T\,d\lambda = \int_{(\mathbb{S}^1)^2}\Psi\,d\lambda.
$$ 
and hence we get 
\begin{equation}\label{10000}
\log m-\log L\leq\int_{(\mathbb{S}^1)^2}\Psi\,d\lambda.
\end{equation}
 Since the two $\log^+$ terms vanish on the unit torus, we arrive at  
\begin{equation}\label{10001}
\int_{(\mathbb{S}^1)^2}\Psi\,d\lambda=\mathscr{M}(F)-\log L,
\end{equation}
where $\mathscr{M}(F)$ denotes the logarithmic two-variable Mahler measure of $F$.

\bigskip

Fix $z$ with $|z|=1$. Then, 
$$
F(z,w)=G(z)w-G^*(z)
$$ 
is a polynomial in $w$. For almost every such $z$, by Jensen's formula, we get, 
$$
\frac{1}{2\pi}\int_0^{2\pi}\log|F(z,e^{i\phi})|\,d\phi=\log\max\{|G(z)|,|G^*(z)|\}.
$$ 
Since $G$ has real coefficients and $|z|=1$, we get 
$$G^*(z)=z^2\overline{G(z)}, \mbox{ and hence } |G^*(z)|=|G(z)|.
$$ 
Consequently, $$\frac{1}{2\pi}\int_0^{2\pi}\log|F(z,e^{i\phi})|\,d\phi=\log|G(z)|$$ for almost every $z$ on the unit circle. Integrating with respect to $z$ on the unit circle, we get  
\begin{equation}\label{99991}
\mathscr{M}(F)= \frac{1}{(2\pi)^2}\int_0^{2\pi}\int_0^{2\pi}\log|F(e^{it},e^{i\phi})|\,d\phi\,dt = \frac{1}{2\pi}\int_0^{2\pi}\log|G(e^{it})|\,dt=\log M(G).
\end{equation}
Therefore, by \eqref{10000}, \eqref{10001} and \eqref{99991}, we get  
$$
\log m - \log L \leq\log M(G)-\log L  \iff  m\leq M(G).
$$ 
This contradicts  the construction of $G$ with $M(G)<m.$ Therefore, 
$$\inf_{P\in \mathscr{P}_{m,a,b}} M(P)>1
$$ 
which proves the theorem. $\hfill\Box$

\section{Proof of Theorem \ref{thm:general-irreducibility}}\label{sec:8}

First we prove the irreducibility of $P_{d,A}$.  By Proposition \ref{prop:root-distribution}, we see that $d-2$ roots of $P_{d,A}$ are on the unit circle and the remaining two roots are off the unit circle, real, negative and reciprocal. 

Set $c=A-1\ge 3$. We first show that no root of unity is a root of $P_{d,A}$.  

Suppose that $\zeta$ is a root of unity satisfying $P_{d,A}(\zeta)=0$. Since $P_{d,A}(1)=2+A(d-1)>0,$ we conclude that 
$\zeta\neq 1$, and therefore $F_{d,c}(\zeta)=0$ where $F_{d,c}(x) = (x-1) P_{d,A}(x) = x^{d+1} +cx^d-cx-1$. Since $F_{d,c}(\zeta) = 0$, by setting $\eta=\zeta^d$, we see that  
$$\eta(\zeta+c)=c\zeta+1,$$ 
and hence 
\begin{equation}\label{eq:root-unity-relation} 
1-\eta\zeta=c(\eta-\zeta). 
\end{equation} 
Since the number $\eta\zeta=\zeta^{d+1}$ is again a root of unity,  \eqref{eq:root-unity-relation} implies that 
$$
1-\eta\zeta\in c\mathcal{O}_{\mathbb{Q}(\zeta)},
$$
where $\mathcal{O}_{\mathbb{Q}(\zeta)}$ is the ring of integers of $\mathbb{Q}(\zeta)$. Then by Lemma~\ref{lem:root-unity}, we conclude that  $\eta\zeta=1$. By \eqref{eq:root-unity-relation}, we obtain $\eta=\zeta$. Thus 
$$
\zeta^2=\eta\zeta=1 \implies \zeta=\pm 1.
$$
Since $P_{d,A}(1)=2+A(d-1)\neq 0$, we conclude that $\zeta = -1$.  Since $d$ is even, 
$$P_{d,A}(-1) = 2+ A\sum_{j=1}^{d-1}(-1)^j=2-A \ne 0,
$$  
because $A\ge 4$. Thus no root of unity is a root of $P_{d,A}$.   

\smallskip

Now suppose that $P_{d,A}(x)=G(x)H(x)$ is a non-trivial factorization in $\mathbb{Z}[x]$. By Gauss's lemma, we may take $G$ and $H$ to be monic. Since $P_{d,A}(0)=1$, their constant terms are $\pm 1$. Consequently, for each factor, the product of the moduli of its roots is $1$.

By Proposition~\ref{prop:root-distribution}, the only roots of $P_{d,A}$ off the unit circle are $\rho$ and $\rho^{-1}$, where $|\rho|>1$. No factor can contain exactly one of these two roots; otherwise, the product of the absolute values of its roots would be either $|\rho|>1$ or $|\rho|^{-1}<1$, contradicting the fact that its constant term has absolute value $1$.

Thus one non-constant factor contains both $\rho$ and $\rho^{-1}$, while the other non-constant factor has all of its roots on the unit circle. By Kronecker's theorem, the latter factor is a product of cyclotomic polynomials. This means that a root of unity is a root of  $P_{d,A}$,   a contradiction. Hence $P_{d,A}$ is irreducible over $\mathbb{Q}$.

\bigskip

Now we compute the Mahler measure of $P_{d,A}$. 

Let $c=A-1$. By Proposition~\ref{prop:root-distribution}, the roots of $P_{d,A}$ consist of $d-2$ roots on the unit circle, one real root, say,  $\rho=-r$ for some real $r>1,$ and its reciprocal $-r^{-1}$. Since $P_{d,A}$ is monic, $M(P_{d,A}) =r.$ 
Substituting $x=-r$ into $F_{d,c}(x)=x^{d+1}+cx^d-cx-1$ and using that $d$ is even, we obtain 
$$
-r^{d+1}+cr^d+cr-1=0.
$$ 
Thus 
\begin{equation}\label{eq:r-equation} 
r^d(r-c)=cr-1. 
\end{equation} 
From the proof of Proposition~\ref{prop:root-distribution}, the root $\rho$ satisfies $\rho<-c$, and hence we get $r>c=A-1.$

\smallskip

For the upper bound, we evaluate $F_{d,c}(x)$ at $x = -(c+1)$ to get
 $$
 F_{d,c}(-(c+1))=-(c+1)^d+c(c+1)-1.
 $$ 
 Since $d\ge 2$, $ (c+1)^d\ge (c+1)^2$, and therefore, 
 $$
 F_{d,c}(-(c+1))\le -(c+1)^2+c(c+1)-1=-c-2<0.
 $$ 
 On the other hand, $F_{d,c}(-c)=c^2-1>0$. Hence we get $-(c+1)<\rho<-c,$ which implies $c<r<c+1.$ Since $c=A-1$, we precisely have 
 $$A-1<M(P_{d,A})<A.
 $$
Since $c < r < c+1$ and by \eqref{eq:r-equation}, we get 
$$0< r-c=\frac{cr-1}{r^d} \le \frac{c(c+1)-1}{c^d}.
$$ 
Because $c\ge 3$, the right-hand side tends to zero as $d\to\infty$. Therefore $r\rightarrow c=A-1,$ and hence 
$$ 
M(P_{d,A})\rightarrow A-1.
$$ The assertion about the infimum follows from the strict lower bound and the convergence.
$\hfill\Box$

\bigskip

\section{Proof of Theorem \ref{thm:support-abmodm}}

We shall start with some lemmas.

Let $d\geq 2, A$ and $B$ be integers and $P_{d, A, B}(x)$ as defined in Theorem \ref{thm:support-abmodm}. 
Define $q(x)=x^2+Ax+B-1\in \mathbb{Z}[x]$ and its degree-two reciprocal reversal $q^*(x)=x^2q(1/x)=(B-1)x^2+Ax+1.$

\begin{lemma}\label{lem:identity}
For every integer $d\ge 2$, and $A, B\in \mathbb{Z}$, we have
$$
(x^2-1)P_{d,A,B}(x) = x^{2d}q(x)-q^*(x).
$$
\end{lemma}

\begin{proof}
Note that  $$(x^2-1)(x+x^3+\cdots+x^{2d-1}) = x^{2d+1}-x \mbox{ and }  (x^2-1)(x^2+x^4+\cdots+x^{2d-2}) = x^{2d}-x^2.
$$
Therefore, we get  
$$
\begin{aligned} 
(x^2-1)P_{d,A,B}(x) &= x^{2d+2}-x^{2d} +A(x^{2d+1}-x) +B(x^{2d}-x^2) +x^2-1\\ 
&= x^{2d+2} +A x^{2d+1} +(B-1)x^{2d} -(B-1)x^2 -Ax -1\\ 
&= x^{2d}(x^2+Ax+B-1) -\bigl((B-1)x^2+Ax+1\bigr)\\ &= x^{2d}q(x)-q^*(x),
\end{aligned}
$$
as desired. 
\end{proof}

\begin{lemma}\label{lem:q-roots}
Assume  integers satisfying $A>B\ge 1$.  Then the polynomial $q(x)=x^2+Ax+B-1$ has two real roots
$\alpha,\beta$ satisfying $-1<\alpha\le 0, \quad -A\le\beta<-A+1.$ In particular, $|\alpha|<1<|\beta|.$
\end{lemma}

\begin{proof}
Suppose first that $B>1$.  Since $q(-1)=B-A<0$ and $q(0)=B-1>0$, the polynomial $q(x)$ has a root $\alpha\in(-1,0).$ Let $\beta$ be the other root.  Since $\alpha+\beta=-A,$ we have $\beta=-A-\alpha.$ Since $-1<\alpha<0$, we get $-A<\beta<-A+1.$
Since $A>B\ge 2$, we have $A\ge 3$, and therefore
$|\beta|>A-1\ge 2.$ 

If $B=1$, then the polynomial $q(x)=x(x+A),$ and hence  we may take $\alpha=0$ and $\beta=-A$, which finishes the lemma.
\end{proof}

\begin{proposition}\label{prop:root-distribution-ab}
Assume that integers $A\geq 4$, $A>B\ge 1$ and $d\ge 2$.  Then the polynomial $P_{d,A,B}(x)$ defined in Theorem \ref{thm:support-abmodm}  has exactly $2d-2$ roots on the unit circle and exactly two roots off the unit circle.  These two roots are real, negative, and reciprocal to each other.
\end{proposition}

\begin{proof}
Let $\alpha,\beta$ be the roots from Lemma~\ref{lem:q-roots}, and put $\gamma=\beta^{-1}.$  Note that 
$|\alpha|<1 \mbox{ and }  |\gamma|<1.$ Then,  we have 
$$q(z)=(z-\alpha)(z-\beta) \mbox{ and } q^*(z)=(1-\alpha z)(1-\beta z).
$$ 
Since $z-\beta=-\beta(1-\gamma z)$ and $1-\beta z=-\beta(z-\gamma)$,  we obtain 
$$\frac{q^*(z)}{q(z)} = \left(\frac{1-\alpha z}{z-\alpha}\right)\left(\frac{1-\beta z}{z-\beta}\right) = \frac{b_\gamma(z)}{b_\alpha(z)},
$$
where $b_\alpha(z)$ corresponds to the Blaschke factor.  By Lemma~\ref{lem:identity}, on $|z|=1$, the equation $(x^2-1)P_{d,A,B}(x)=0$ is equivalent to 
$$
z^{2d}q(z)=q^*(z), \mbox{ or } z^{2d}\frac{b_\alpha(z)}{b_\gamma(z)}=1.
$$
Write $b_\alpha(e^{it})=e^{i\theta_\alpha(t)}, b_\gamma(e^{it})=e^{i\theta_\gamma(t)},$ and set 
$$
\Psi(t) = 2dt+\theta_\alpha(t)-\theta_\gamma(t).
$$
Then  by Lemma \ref{lem:blaschke-derivative}, 
$$
\Psi'(t) = 2d + \frac{1-\alpha^2}{1-2\alpha\cos t+\alpha^2} - \frac{1-\gamma^2}{1-2\gamma\cos t+\gamma^2},
$$
where the middle term is positive.  Also, we have $\displaystyle\frac{1-\gamma^2}{1-2\gamma\cos t+\gamma^2} \le \frac{1+|\gamma|}{1-|\gamma|}.$ By Lemma~\ref{lem:q-roots}, $|\beta|>A-1$ when $B>1$, while $|\beta|=A$ when $B=1$.  In either case,
$$
|\gamma|<\frac{1}{A-1} \mbox{ and hence } \frac{1+|\gamma|}{1-|\gamma|} < \frac{A}{A-2}.
$$ 
By hypothesis  $A\ge 4$ and hence $\displaystyle\frac{A}{A-2}\le 2.$ Thus, it follows that $\Psi'(t)>2d-2>0.$ Hence the function $\Psi$ is strictly increasing. 

Note that both the Blaschke arguments increase by $2\pi$,  and hence their total changes cancel: $$\Psi(2\pi)-\Psi(0)=4\pi d.$$ Therefore  $\Psi(t)\in2\pi\Z$ has exactly $2d$ solutions on $0\le t<2\pi$. Hence $(x^2-1)P_{d,A,B}(x)$ has exactly $2d$ distinct roots on the unit circle.  Moreover,
$$
P_{d,A,B}(1)=2+dA+(d-1)B>0,\qquad
P_{d,A,B}(-1)=2-dA+(d-1)B<0,
$$
since $A>B\geq1$ and $d\geq2$. Thus $P_{d,A,B}(\pm1)\neq0$.  Since the roots $1$ and $-1$ arise from $x^2-1$, we conclude that $P_{d,A,B}$ has exactly $2d-2$ roots on the unit circle.

Since $P_{d,A,B}$ has degree $2d$, precisely two roots remain. Since $P_{d,A,B}$ is reciprocal, these two roots are reciprocal.
To see that they are real and negative, put $t=A-1\geq3$ and $F(x)=(x^2-1)P_{d,A,B}(x)=x^{2d}q(x)-q^*(x)$. As $q(-t)=B-A$ and $q^*(-t)=(B-1)t^2-At+1$, we have
$$
F(-t)\leq-t^{2d}+At-1<0.
$$
On the other hand, $F(-A)=A^{2d}(B-1)-(B-2)A^2-1>0$ for every $B\geq1$ under the hypotheses, as is checked by separating the cases $B=1$, $B=2$, and $B\geq3$. Hence $P_{d,A,B}$ has a real root in $(-A,-(A-1))$, and the other exceptional root is its reciprocal, proving the assertion.
\end{proof}

\begin{proposition}\label{prop:no-cyclotomic}
If $A>B\geq1$ are integers such that ${\rm{gcd}}(A-1, B-1)\geq 3$, then no root of the polynomial $P_{d,A,B}$ is a root of unity.
\end{proposition}

\begin{proof}
Let $g=\gcd(A-1,B-1)\ge 3.$ Then $A\equiv B\equiv 1\pmod g.$ Suppose that $\zeta$ is a root of unity satisfying $P_{d,A,B}(\zeta)=0.$ By reducing modulo $g$, we get  $P_{d,A,B}(x) \equiv 1+x+x^2+\cdots+x^{2d}\pmod{g}.$ Therefore 
$1+\zeta+\zeta^2+\cdots+\zeta^{2d} \in g\mathcal{O}_{\mathbb{Q}(\zeta)}.$  Multiplying by $\zeta-1$, we obtain 
$$
\zeta^{2d+1}-1 \in g\mathcal{O}_{\mathbb{Q}(\zeta)}.$$ Since $\zeta^{2d+1}$ is a root of unity, by  Lemma~\ref{lem:root-unity}, we arrive at  $\zeta^{2d+1}=1.$

Since $P_{d,A,B}(1) = 2+dA+(d-1)B>0$, we also have $\zeta\neq 1$.   Thus $\zeta^{2d}=\zeta^{-1}.$ Now,  substituting $x=\zeta$ into $x^{2d}q(x)-q^*(x)=0$ gives  
$$
\zeta^{-1}q(\zeta)-q^*(\zeta)=0, \mbox{ or } q(\zeta)-\zeta q^*(\zeta)=0.
$$ 
A direct calculation shows that $$q(x)-xq^*(x) = -(x-1) \left((B-1)x^2+(A+B-2)x+B-1\right).$$
Since $\zeta\neq 1$, $$(B-1)\zeta^2+(A+B-2)\zeta+B-1=0.$$ Dividing by $\zeta$, we obtain 
$$
(B-1)(\zeta+\zeta^{-1})+A+B-2=0.
$$
If $B=1$, the above equation implies that $A-1=0$, a contradiction.

If $B>1$, then 
$$\zeta+\zeta^{-1} = -\frac{A+B-2}{B-1}.
$$ 
Since  $|\zeta|=1$, we have $-2\le \zeta+\zeta^{-1}\le 2.$ But $A>B$ implies $A+B-2>2(B-1),$ and therefore
$$\frac{A+B-2}{B-1}>2,$$ a contradiction.  Thus $P_{d,A,B}(x)$ has no root of unity.
\end{proof}

\noindent{\bf Proof of Theorem \ref{thm:support-abmodm}.} (1) We now prove that $P_{d, A, B}(x)$ is irreducible over $\mathbb{Q}$. 

Suppose
$$P_{d,A,B}(x)=G(x)H(x)$$
is a nontrivial factorization in $\Z[x]$.  By Gauss's lemma, we may
take $G$ and $H$ to be monic.  Since
$$P_{d,A,B}(0)=1,$$
their constant terms are $\pm1$.  Thus, for each factor, the product
of the moduli of all its roots equals $1$.

By Proposition~\ref{prop:root-distribution-ab}, the only roots of
$P_{d,A,B}$ off the unit circle are $\rho, \rho^{-1}$ with $|\rho|>1.$  A factor cannot contain exactly one of these roots, because then the product of the moduli of its roots would be either $|\rho|>1$ or $|\rho|^{-1}<1,$ rather than $1$.

Consequently, one non-constant factor contains both roots off the unit circle,
while the other non-constant factor has all its roots on the unit
circle.  By Kronecker's theorem, a monic integer coefficient polynomial all of
whose roots lie on the unit circle is a product of cyclotomic
polynomials.  Therefore, the roots of this second factor are roots of unity,   which contradicts Proposition~\ref{prop:no-cyclotomic}.  Hence $P_{d,A,B}(x)$ is irreducible over $\mathbb{Q}$. 

\bigskip

(2) Now we compute the Mahler measure of $P_{d,A,B}(x)$ as follows. 

Let $F_{d,A,B}(x) = (x^2-1)P_{d,A,B}(x) = x^{2d}q(x)-q^*(x).$ Then by Proposition~\ref{prop:root-distribution-ab}, $P_{d,A,B}$ has one root $\rho=-r$ outside the unit circle, with $r>1$, while its other off-unit-circle root is $-r^{-1}$.  All the remaining roots lie on the unit circle. Therefore 
$$M(P_{d,A,B})=r.
$$ 
We first locate $r$.  Put $t=A-1.$ Since 
$$q(-t) = t^2-At+B-1 = B-A,$$
we have 
$$
F_{d,A,B}(-t) = t^{2d}(B-A) - \bigl((B-1)t^2-At+1\bigr) 
= -(A-B)t^{2d} -(B-1)t^2 +At -1. 
$$ 
Since $A-B\ge 1$ and $B-1\ge 0$, we get
$$
F_{d,A,B}(-t) \le -t^{2d}+At-1.
$$ 
Since $A=t+1$, we see that $At-1=t^2+t-1.$  Since $t=A-1\ge 3$ and $d\ge 2$, we get 
$$t^{2d}\ge t^4>t^2+t-1.
$$ 
Hence  $F_{d,A,B}(-(A-1)) = F_{d, A, B}(-t)\leq -t^{2d} + t^2 +t -1 <0.$ Next, 
$$
q(-A)=B-1, \mbox{ and }  q^*(-A) = (B-1)A^2-A^2+1 = (B-2)A^2+1.
$$ 
Thus 
$$F_{d,A,B}(-A) = A^{2d}(B-1)-(B-2)A^2-1.
$$ 

If $B=1$, then $F_{d,A,1}(-A)=A^2-1>0.$ 

If $B=2$, then $F_{d,A,2}(-A)=A^{2d}-1>0.$

If $B\ge 3$, then $A^{2d}(B-1) \ge A^4(B-1) > A^2(B-2)+1,$ and  again $F_{d,A,B}(-A)>0.$ Therefore $F_{d,A,B}$ has a real root in $(-A,-(A-1)).$ This root is $\rho=-r$, and hence we get
$$A-1<r<A \mbox{ and hence  we get } A-1<M(P_{d,A,B})<A.
$$ 
For the limiting value, the equation $F_{d,A,B}(-r)=0$ implies 
$$
r^{2d}\bigl(r^2-Ar+B-1\bigr) = (B-1)r^2-Ar+1.
$$ 
Since $A-1<r<A,$ the right-hand side remains bounded in absolute value by a constant depending only on $A$ and $B$, while 
$$r^{2d}\ge(A-1)^{2d}\rightarrow\infty \mbox{ as } d\to\infty.
$$ 
Therefore as $d\to\infty$, we see that $r^2-Ar+B-1\rightarrow 0.$ Now, note that every accumulation point of $r$ must therefore be a positive root of $x^2-Ax+B-1=0.$ The two roots are 
$$\frac{A\pm\sqrt{A^2-4(B-1)}}{2}.
$$ 
The smaller root lies in $[0,1)$, while $r>A-1$.  Hence the only possible limit is the larger root which is 
$$R_{A,B} = \frac{A+\sqrt{A^2-4(B-1)}}{2}.
$$
Thus $$M(P_{d,A,B})=r\rightarrow R_{A,B}.
$$
This proves the theorem. $\hfill\Box$


\bigskip

\noindent{\bf Acknowledgments.} This work was done while the first author was visiting Harish-Chandra Research Institute (HRI).    He is   thankful to HRI for providing the ambient research atmosphere. Both the authors are very grateful  to Professor B. Ramakrishnan for his constant support and  encouraging words.

\end{document}